\documentclass[11pt,a4paper]{article}

\usepackage[utf8]{inputenc}
\usepackage[T1]{fontenc}
\usepackage{amsmath}
\usepackage{amsthm}          
\usepackage{newtxtext,newtxmath}  
\usepackage[margin=2.5cm]{geometry}
\usepackage{mathtools}
\usepackage[round]{natbib}
\usepackage{graphicx}
\usepackage{url}
\usepackage{xcolor}
\usepackage{float}
\usepackage{enumitem}
\usepackage{tikz}
\usepackage[colorlinks=true,linkcolor=blue!60!black,citecolor=blue!60!black,urlcolor=blue!60!black]{hyperref}

\theoremstyle{plain}
\newtheorem{theorem}{Theorem}
\newtheorem{lemma}[theorem]{Lemma}
\theoremstyle{definition}
\newtheorem{definition}[theorem]{Definition}
\newtheorem{step}{Step}
\theoremstyle{remark}
\newtheorem{remark}[theorem]{Remark}

\newenvironment{keywords}{\par\medskip\noindent\textbf{Keywords:}\ }{\par}

\newcommand{\X}{\mathcal{X}}
\newcommand{\Mset}{\mathcal{M}}
\newcommand{\PX}{\mathcal{P}(\X)}
\newcommand{\R}{\mathbb{R}}
\newcommand{\KL}{\mathrm{KL}}
\newcommand{\Df}{D_f}
\newcommand{\Ef}{\mathcal{E}_f}
\newcommand{\Ekl}{\mathcal{E}_{\KL}}
\newcommand{\dd}{\,\mathrm{d}}
\newcommand{\eqdef}{\mathrel{\overset{\text{def}}{=}}}
\DeclareMathOperator*{\argmin}{arg\,min}
\DeclareMathOperator{\spn}{span}
\newcommand{\MT}[1]{M^{\mathcal{T}}_{#1}}

\title{On Minimum Divergence for Distribution Reconstruction}
\author{%
  Justin Bec\thanks{Equal contribution.}
  \qquad Thomas Le Corre\footnotemark[1]
  \qquad Aur\'elien Garivier\\[8pt]
  \small Unit\'e de Math\'ematiques Pures et Appliqu\'ees (UMPA)\\
  \small ENS de Lyon, France
}
\date{}

\begin{document}
	\maketitle
	
	\begin{abstract}
	We study the possibility of generalizing  maximum-entropy reconstruction, which approximately recovers a probability distribution from finitely many (generalized) moments by minimizing the Kullback--Leibler (KL) divergence to a reference measure.
	The reconstruction lies in an exponential family and is also the best KL approximation of the unknown distribution within that family: moment matching and information projection coincide. Replacing KL by a smooth $f$-divergence, the reconstruction still exists in an associated family, but the best approximation property does not. Specifically, a family is called \emph{entropy-adapted} if the projection of any distribution matches its moments. We show that, up to equivalence, KL is the only $C^2$ $f$-divergence that is entropy-adapted, whatever the choice of the generalized moments. We illustrate the consequences of this failure numerically.
	\end{abstract}
	\section{Introduction}
	\subsection{Moment reconstruction problems}
	
	Recovering a probability distribution from incomplete information is a classical problem in probability, statistics, and applied mathematics. In many situations, the distribution of interest is not directly accessible: only a finite collection of summary quantities, such as means, variances, or other expectations, can be observed or estimated. The \emph{moment reconstruction problem} consists in using this partial information to estimate the corresponding probability distribution. While classical mathematical foundations go back to the Hamburger, Stieltjes, and Hausdorff moment problems \citep{shohat1943problem,akhiezer1965moment}, this problem appears to be of particular importance in \emph{Distributional Reinforcement Learning} (DistRL), where finite representations of return distributions that are mixed and convolved in Bellman iterations are a crucial bottleneck~\citep{bellemare2017}. It has been shown recently that maximum-entropy reconstruction outperforms the traditional quantile or expectile approaches by orders of magnitude~\citep{bringas2026momentmatching}. 
	
	Moment information is particularly natural because expectations of prescribed functions can often be estimated even when the underlying distribution cannot be observed in full. Classical polynomial moments are widely used in probability and statistics, while they are shown in~\citet{Rowland2019,marthe2023} to be the only possible choice in discounted DistRL. Meanwhile, generalized moments make it possible to encode a much broader range of features. Depending on the choice of the functions being averaged, they may describe physical observables, marginal statistics, spectral coefficients, option prices, or other linear measurements of an unknown distribution \citep{karlin1966tchebycheff}. Reconstruction from moments thus appears in many fields. In statistics, the method of moments is one of the oldest approaches to parameter estimation, and moment conditions form the basis of the generalized method of moments \citep{pearson1894iii,hansen1982large}. Related questions also arise in spectral estimation, signal processing, inverse problems, uncertainty quantification \citep{burg1975maximum,mead1984maximum,borwein1991convergence} and more recently in distributional reinforcement learning \citep{bringas2026momentmatching}.
	
	A central difficulty is that finitely many moments typically do not determine a unique probability distribution\footnote{In fact, even the complete infinite sequence of polynomial moments may not be sufficient if the support or additional growth conditions are not specified~\citep{shohat1943problem,akhiezer1965moment}.}. With only finitely many moments, the problem is generically under-determined: infinitely many distributions are compatible with them. Moment reconstruction therefore involves two distinct questions:
	\begin{itemize}
		\item Feasibility: are the prescribed moments compatible with at least one
		probability distribution?
		\item Selection: which distribution should be chosen among all compatible candidates?
	\end{itemize}
	
	Several selection principles have been proposed. Depending on the application, one may for example impose a parametric model and minimize a regularization functional. Another particularly influential approach is to choose the compatible distribution that minimizes an information criterion relative to a prescribed reference measure. This leads to maximum-entropy and minimum-divergence reconstruction, which provide the starting point of the present work \citep{jaynes1957information,csiszar1975divergence,borwein1991convergence}.
	
	\subsection{Maximum-entropy reconstruction}
	
	A classical way to resolve the ambiguity inherent in moment reconstruction is the \emph{maximum-entropy principle}. Introduced in statistical mechanics by \citet{jaynes1957information,jaynes1957informationII}, this principle prescribes selecting the probability distribution that is compatible with the known moments without adding further information. Indeed, the principle was originally motivated by the interpretation of entropy as a measure of missing information: the chosen distribution is taken to be the one compatible with the moment constraints that has the maximal Shannon entropy, or equivalently the minimal Kullback--Leibler (KL) divergence to the uniform measure. Relative versions of the principle allow another reference distribution to be incorporated, leading to the closely related method of minimum discrimination information developed by Kullback and further studied by Csiszár~\citep{kullback1951information,kullback1959information,csiszar1975divergence}.
	Maximum-KL-entropy methods have since become standard tools in probability, statistics, information theory and inverse problems. They have been used for the reconstruction of distributions from moments \citep{mead1984maximum,bandyopadhyay2005maximum}, for statistical inference under linear constraints \citep{csiszar1975divergence,csiszar1984sanov}, and for the analysis of under-determined inverse problems \citep{borwein1991convergence,borwein1994entropy}. More broadly, entropy minimization plays an important role in convex optimization, large-deviation theory, statistical physics and computational optimal transport \citep{dembo1998large,peyre2019computational}.
	
	Classically, maximum-entropy methods are hence based on Shannon entropy, or equivalently on KL divergence. But any other notion of divergence, with its own geometry and computational properties, is a legitimate candidate for alternative procedures. To avoid any misunderstanding, we will hence in the sequel rather speak of maximum-KL-entropy methods: the purpose of this article is precisely to investigate by which other divergences KL can be replaced. 
	
	One of the main reasons for the success of the maximum-KL-entropy principle is that moment constraints lead naturally to exponential families. These families form one of the central classes of statistical models and include many classical distributions. Their statistical and geometric properties have been extensively studied, notably in the theories of exponential families and information geometry \citep{brown1986fundamentals,barndorff1978information,amari2000methods,
		amari2016information}. In particular, exponential families provide a finite-dimensional representation of a reconstruction problem that is initially posed over an infinite-dimensional space of probability measures.
	
	Thanks to the geometry of the Kullback--Leibler divergence~\citep{csiszar1975divergence,csiszar2004information,amari2016information}, this reduction is accompanied by a remarkable projection property: the member of the exponential family that matches the prescribed moments is also the Kullback--Leibler best approximation of the unknown distribution within that family. This correspondence between moment matching, entropy minimization and information projection is particularly valuable in reconstruction problems: the best approximation is hence efficiently computed as a simple convex optimization program.
	
	The coincidence of these three viewpoints (maximum-entropy reconstruction, modelling by exponential families, and projection by moment matching) is key to the success of maximum-KL-entropy reconstruction. But are there even better choices? Does this property  also hold for other divergences? This paper investigates whether it is a distinctive feature of Kullback--Leibler geometry, and what may fail for other choices.
	
	\subsection{\(f\)-divergences}
	
	The Kullback--Leibler divergence is only one among many possible measures of discrepancy between probability distributions. A broad and widely used class is provided by the \(f\)-divergences introduced by Csiszár and, independently, by Ali and Silvey \citep{csiszar1967information,alibsilvey1966general}. This class contains, among others, the Kullback--Leibler and reverse Kullback--Leibler divergences, the Hellinger distance, the Pearson and Neyman \(\chi^2\)-divergences, and the total variation distance. Owing to their general form and their contraction under measurable maps, \(f\)-divergences have become central tools in information theory, statistics, machine learning and hypothesis testing \citep{liese2006divergences,csiszar2004information,polyanskiy2024information}. Replacing the Kullback--Leibler divergence by another \(f\)-divergence in an entropy-minimization problem leads naturally to generalized notions of entropy and to deformed analogues of exponential families. Such constructions appear in several branches of generalized statistical mechanics and information geometry, notably through power-law, \(q\)-exponential and \(\alpha\)-families \citep{naudts2004estimators,amari2016information}.
	
	Projection theorems and Pythagorean inequalities are likewise available beyond the classical entropy setting, particularly for Bregman divergences and for various generalized geometries \citep{bregman1967relaxation,amari2009alpha}. Alternative divergences may also admit their own generalized linear families, deformed moment constraints, or dual projection structures \citep{sundaresan2007guessing,kumar2015projection,kumar2015reverse,wong2018logarithmic}. However, the existence of a generalized entropy family does not automatically imply that its best approximation can be recovered by matching the initially observed moments. This distinction motivates the central question of the present work. Suppose that the same \(f\)-divergence is used both to select a distribution under moment constraints and to measure the approximation error with respect to the resulting generalized family. For which generators does the resulting projection coincide, uniformly, with ordinary moment matching? Equivalently, which \(f\)-divergences preserve the classical compatibility between linear constraints, entropy minimization and projection onto the associated family?

	\begin{figure}[ht]
		\centering
		\resizebox{\textwidth}{!}{%
		\begin{tikzpicture}[
			scale=1,
			every node/.style={font=\small},
			>=latex
			]
			
			\draw[thick, rounded corners=8pt] (-6,-2.2) rectangle (8,3.2);
			\node[anchor=north west,font=\bfseries] at (-5.8,3.0) {Divergences};
			
			\draw[thick, blue] (-1.6,0.4) ellipse (1.6 and 1.4);
			\node[blue!80!black] at (-1.6,0.75) {$C^2$ $f$-divergences};
			
			\draw[thick, teal!70!black] (1.7,0.4) ellipse (1.7 and 1.3);
			\node[teal!70!black] at (1.7,0.75) {Bregman divergences};
			
			\draw[thick, orange!85!black] (2.9,0.4) ellipse (2.9 and 2.5);
			\node[orange!85!black,align=center] at (2.9,2.1) {Entropy-adapted \\ divergences};
			
			\fill (0,0.4) circle (2pt);
			\node[left=3pt] at (0,0.4) {$\KL$};
			
			\fill (3.4,-0.8) circle (2pt);
			\node[below right=0pt] at (3.4,-0.8) {$\KL^2$};
			
		\end{tikzpicture}}
		\caption{Schematic relation between $C^2$ $f$-divergences, Bregman divergences, and the entropy-adapted divergences (Definition~\ref{def:adapted}). On probability distributions, the classes of $f$-divergences and Bregman divergences intersect only at $\KL$ (up to equivalence), which is a classical result \citep{amari2009alpha}. The projection property is verified for every Bregman divergence but is weaker than being Bregman, as illustrated by $\KL^2$. More details are presented in Appendix~\ref{s:AppA}.}
		\label{fig:bregman-fdiv-summary}
	\end{figure}
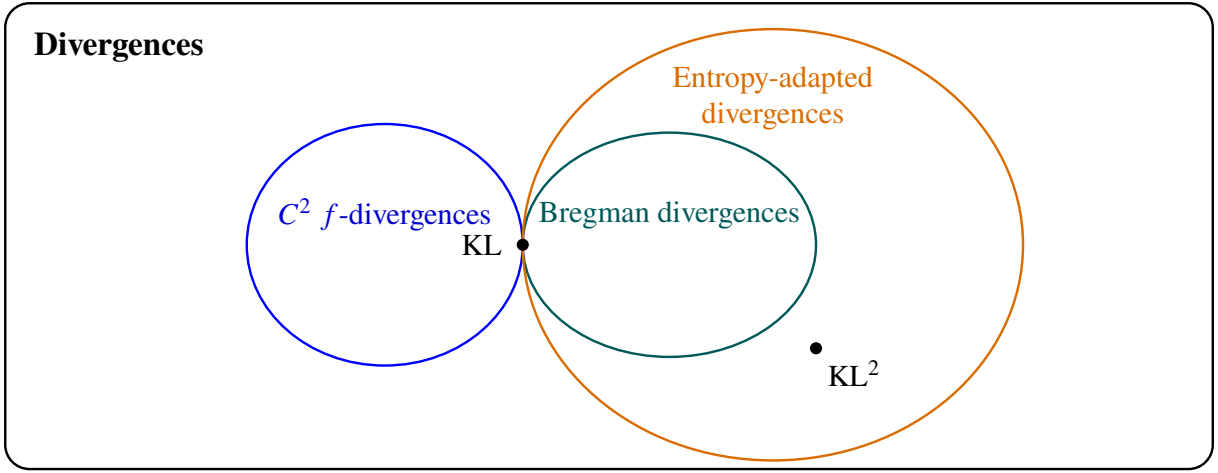

	Several existing results already single out the Kullback--Leibler divergence through other structural properties. In particular, among divergences between probability distributions, it is known to be the only one that is both an \(f\)-divergence and a Bregman divergence \citep{frigyik2008functional,amari2009alpha}; on positive measures, this property is shared by the whole class of \(\alpha\)-divergences. Other characterizations rely on additivity, sufficiency, locality or consistency of inference rules \citep{csiszar1991why,shore1980axiomatic}. The characterization developed here is of a different nature: it is based directly on the compatibility of maximum-entropy reconstruction with projection by ordinary moment matching. As illustrated by elementary counterexamples for the reverse Kullback--Leibler and Hellinger divergences in Figure~\ref{fig:counterexamples}, this compatibility fails for standard alternatives to relative entropy. The purpose of this paper is to show that this failure is not accidental. Under the assumptions specified below, the Kullback--Leibler divergence is, up to the usual equivalence of generators, the unique \(f\)-divergence for which the entropy reconstruction and the best approximation within the associated family are universally connected by ordinary moment matching.
	
	\subsection{Contributions and Organisation of the Paper}
	
	Maximum-entropy reconstruction under moment constraints is usually tied to KL, where it coincides with moment matching and with information projection. We ask whether this compatibility survives when KL is replaced by a general $f$-divergence, and show that it does not:
	
	\begin{itemize}
		\item We formalize the compatibility above as the property of being \emph{$\mathcal T$-entropy-adapted}.
		\item Our main result (Theorem~\ref{thm:main}) is that, for \emph{every} free family $\mathcal T$, KL is,
		up to equivalence, the \emph{only} $C^2$ $f$-divergence that is $\mathcal T$-entropy-adapted: no choice of
		moment functions, however special, makes another $f$-divergence compatible with moment matching. We prove
		this result by writing the Lagrange optimality condition of the $\Df$-projection at a well-chosen point of
		$\Ef$ and exploiting it on finite configurations.
		\item We discuss two consequences: in moment-based DistRL, the refined convergence guarantees of maximum-entropy reconstruction are specific to KL; in classical statistics, the coincidence of maximum likelihood, KL projection and moment matching in exponential families is likewise not generic to other $f$-divergences.
	\end{itemize}
	
	The rest of the paper is organized as follows. Section~\ref{s:Preliminaries} sets up the moment
	reconstruction problem, recalls the classical KL/exponential-family picture, and introduces the
	$f$-exponential family $\Ef$ together with the $\mathcal T$-entropy-adapted property and our main
	theorem. Section~\ref{s:Connections} discusses the two applications sketched above, Section~\ref{s:Proof}
	contains the proof of Theorem~\ref{thm:main}, and Section~\ref{s:Num} illustrates the result numerically. Appendix~\ref{s:AppA} relates it to the classical KL/Bregman characterization.
	
	\section{Formal setting and notation}
	\label{s:Preliminaries}
	
	Let $\X \subseteq \R$ be a measurable set, finite or not, equipped with a reference probability measure $\lambda$, and let
	\[
	\mathcal{T} = (T^{0}, T^{1}, \dots, T^{k})
	\]
	be a finite family of measurable functions $T^{j} : \R \to \R$ with $T^0\equiv1$, which we assume to be \emph{free}, i.e.\ linearly independent as functions on $\R$ (their restrictions to a given $\X$ need not be). We write $\Mset$ for the set of probability measures on $\X$ that are equivalent to $\lambda$, and for an integrable function $g$ and $\nu \in \Mset$ we abbreviate $\nu(g) = \int_{\X} g \dd\nu$. Given an unknown probability measure $\nu_0$, we assume that only the moment vector
	\[
	\nu_0(\mathcal{T}) = \bigl(\nu_0(T^{0}), \nu_0(T^{1}), \dots, \nu_0(T^{k})\bigr) \in \R^{k+1}\,
	\]
	is available.
	
	\begin{remark}
		When $\X = [0,1]$ and $T^{j}(x) = x^j$ these are the usual moments of $\nu_0$, which is where the terminology comes from. The moment functions are defined on the whole of $\R$ because the same family $\mathcal{T}$ will be used on several spaces $\X$, in particular on finite ones (Definition~\ref{def:adapted}); nothing below depends on the structure of $\R$, which could be replaced by any infinite set.
	\end{remark}
	
	Except for very special spaces $\X$, this inverse problem admits infinitely many solutions as the set of measures consistent with the data,
	\[
	\MT{\nu_0} \eqdef \bigl\{\, \mu \in \Mset : \mu(T^{j}) = \nu_0(T^{j})\ \text{for all } 0 \le j \le k \,\bigr\},
	\]
	is an infinite-dimensional convex set, cut out by $k$ affine constraints, and we need some principle to single out one representative. One of the most influential principles for resolving this ambiguity is the maximum-entropy principle.
	
	\subsection{Maximum entropy and the exponential family}
	
	\paragraph{The maximum-entropy principle.}
	The classical answer is to pick, among all measures compatible with
	the data, the closest to the reference $\lambda$. Measuring closeness to $\lambda$
	by the Kullback--Leibler divergence,
	\[
	\KL(\nu, \lambda) = \int_{\X} \ln\!\frac{\dd\nu}{\dd\lambda} \dd\nu,
	\]
	the reconstruction is
	\begin{equation}\label{eq:maxent}
		\nu^\star \;=\; \argmin_{\nu \in \MT{\nu_0}} \KL(\nu, \lambda)\,.
	\end{equation}
	This is a convex problem: $\nu \mapsto \KL(\nu,\lambda)$ is strictly convex and $\MT{\nu_0}$ is
	convex, so its solution, whenever it exists, is unique.
	When $\lambda$ is the uniform measure, $ \KL(\nu, \lambda) = \mathrm{Ent}(\lambda) - \mathrm{Ent}(\nu)$ where $\mathrm{Ent}$ denotes the Shannon entropy, hence the name of maximum entropy reconstruction.
	
	\paragraph{The exponential family.}
	A Lagrange computation on the affine constraints $\nu(T^{j}) = \nu_0(T^{j})$ shows that the minimizer
	of~\eqref{eq:maxent} has a density of the form
	$\dd\nu^\star/\dd\lambda \propto \exp(\theta^{\top} \mathcal{T})$ for some $\theta \in \R^{k+1}$. In other
	words, $\nu^\star$ lies in the \emph{exponential family}
	\[
	\Ekl \;=\; \Bigl\{\, \nu_\theta \in \Mset : \tfrac{\dd\nu_\theta}{\dd\lambda} = \exp(\theta^{\top} \mathcal{T}),\ \theta \in \R^{k+1},\ \nu_\theta(\X) = 1 \,\Bigr\}\,.
	\]
	As the data $\nu_0(\mathcal{T})$ ranges over all attainable values, the corresponding minimizers $\nu^\star$
	sweep out exactly $\Ekl$: the exponential family is the locus of all maximum-entropy
	reconstructions.
	The gain is that, unlike the infinite-dimensional sets $\MT{\nu_0}$, the family $\Ekl$ is
	finite-dimensional. It is faithfully parametrized by $\theta$, which after normalization
	effectively ranges over a convex subset of $\R^k$. Reconstructing $\nu_0$ from its moments thus collapses to
	producing a single vector $\theta$, and this is precisely why one chooses to work inside $\Ekl$.
	
	\begin{remark}
		The picture is the one familiar from information geometry: the constraints defining $\MT{\nu_0}$
		are linear in the density, making it a mixture-flat (\(m\)-flat) affine set, whereas the
		log-densities of $\Ekl$ are affine in $\mathcal{T}$, making it an exponential-flat (\(e\)-flat) set. A
		one-parameter exponential family ($k=1$) is literally an \(e\)-geodesic; for general $k$ it is an
		\(e\)-flat submanifold, and it is in this loose sense that we will speak of $\Ekl$ as the
		maximum-entropy ``geodesic''.
	\end{remark}
	
	\subsection{Projection equals moment matching}
	
	As the reconstruction belongs to $\Ekl$, one may consider the following two questions: \begin{itemize}
		\item[\textbf{(P1)}] \emph{Moment matching.} Find the point of $\Ekl \cap \MT{\nu_0}$, i.e. the
		element of $\Ekl$ that reproduces the observed moments.
		\item[\textbf{(P2)}] \emph{Best approximation.} Find
		$\displaystyle \argmin_{\nu_\theta \in \Ekl} \KL(\nu_0, \nu_\theta)$, i.e. the element of
		$\Ekl$ that is closest to $\nu_0$.
	\end{itemize}
	These are not on the same footing. Problem (P1) depends on $\nu_0$ only through the observed
	moments $\nu_0(\mathcal{T})$, so it is something we can actually solve. Problem (P2) is the one we would
	really like to solve (it returns the best exponential-family approximation of $\nu_0$) but
	it is out of reach on its own, since evaluating $\KL(\nu_0, \cdot)$ requires the very measure
	$\nu_0$ we are trying to reconstruct. What saves us is that, fortunately, the two coincide: the solution of the feasible problem
	(P1) is also the solution of the (a priori infeasible) problem (P2). Problem (P1) admits at most one solution
	$\nu_{\bar\theta} \in \Ekl \cap \MT{\nu_0}$. It always exists when $\X$ is finite: the moment vector of a
	measure $\nu_0$ with full support lies in the relative interior of the convex hull of
	$\{\mathcal{T}(x)\}_{x \in \X}$, which is exactly the image of the moment map
	$\theta \mapsto \nu_\theta(\mathcal{T})$ of $\Ekl$ \citep[see e.g.][]{barndorff1978information,brown1986fundamentals}.
	On infinite spaces it may fail to exist (see Remark~\ref{rem:finite-spaces}). Whenever it exists, the
	following classical identity shows that this
	same $\nu_{\bar\theta}$ minimizes $\KL(\nu_0, \cdot)$ over $\Ekl$ \citep{csiszar2004information}.
	
	\begin{lemma}[Pythagorean identity]\label{lem:pyth}
		For every $\nu_\theta \in \Ekl$,
		\begin{equation}\label{eq:pythagore}
			\KL(\nu_0, \nu_\theta) \;=\; \KL(\nu_0, \nu_{\bar\theta}) + \KL(\nu_{\bar\theta}, \nu_\theta)\,.
		\end{equation}
	\end{lemma}
	
	Since the second term of~\eqref{eq:pythagore} is nonnegative and vanishes only at
	$\nu_\theta = \nu_{\bar\theta}$, the $\KL$-projection of $\nu_0$ onto $\Ekl$ is nothing but the
	element of $\Ekl$ that matches the moments of $\nu_0$: the only way to reach the best
	approximation (P2) is through the moment-matching problem (P1).
	This property is illustrated in Figure~\ref{KL}.
	
	\begin{remark}
		Concretely, (P1) is nothing but a finite-dimensional system of $k$ scalar moment equations in the
		unknown $\theta$, namely $\nu_\theta(\mathcal{T}) = \nu_0(\mathcal{T})$. What could have been an optimization over an
		infinite-dimensional space of measures becomes the resolution of $k$ scalar equations.
	\end{remark}
	
	The consequence that matters most to us is the inequality
	\begin{equation}\label{eq:bound}
		\KL(\nu_0, \nu_{\bar\theta}) \;\le\; \KL(\nu_0, \nu_\theta)
		\qquad \text{for every } \nu_\theta \in \Ekl\,.
	\end{equation}
	Any $\nu_\theta \in \Ekl$ has $\ln \dd\nu_\theta/\dd\lambda = \theta^{\top} \mathcal{T}$, so choosing such a
	comparison point amounts to choosing a linear combination of the basis functions $T^{1}, \dots, T^{k}$
	to approximate the log-density $\ln \dd\nu_0/\dd\lambda$. The reconstruction error
	$\KL(\nu_0, \nu_{\bar\theta})$ is then bounded by the right-hand side of~\eqref{eq:bound}, hence
	controlled entirely by how well $\ln \dd\nu_0/\dd\lambda$ can be approximated by such linear
	combinations, without ever computing $\nu_{\bar\theta}$ itself. This is one of the main
	mechanisms behind the moment-matching reconstruction bounds.
	
	\begin{figure}[H]
		\centering
		\includegraphics[width=0.48\textwidth]{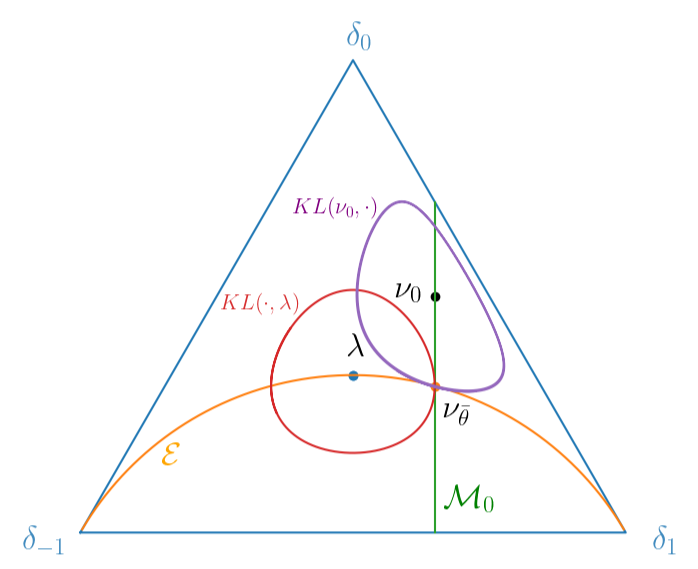}\hfill
		\includegraphics[width=0.48\textwidth]{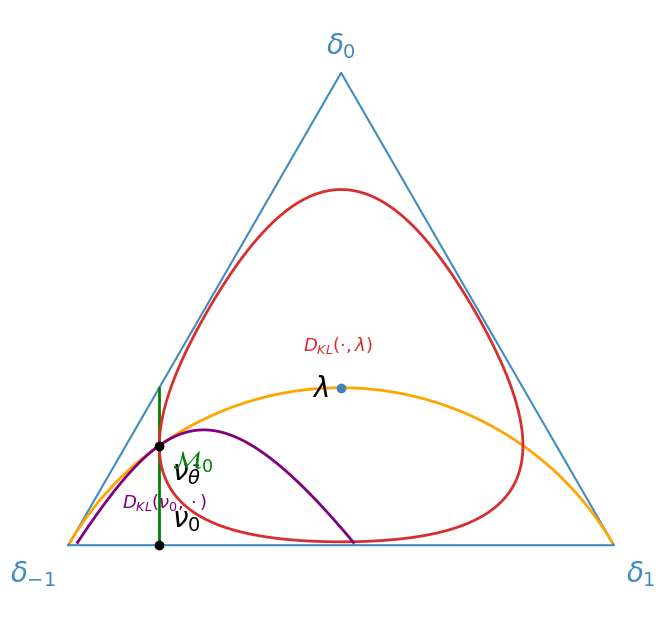}
		\caption{An example with $\X = \{-1,0,1\}$ and the free family $(1,x)$. The exponential family
			$\mathcal{E} = \Ekl$ is in orange and $\mathcal{M}_0 = \MT{\nu_0}$ is in green. The point
			$\nu_{\bar\theta}$ simultaneously satisfies (P1) and (P2).}
		\label{KL}
	\end{figure}
	
	\subsection{Beyond Kullback--Leibler}
	
	The specific choice of Kullback--Leibler among all possible $f$-divergences appeared above
	in two different places: to define what ``closest to $\lambda$'' means in~\eqref{eq:maxent} (the
	inductive bias of the reconstruction), and what ``closest to $\nu_0$'' means in (P2) (the quantity
	we ultimately bound). Neither use singles out KL a priori and any divergence could stand in
	either place. It is therefore natural to ask how much of the construction survives across this
	larger family, and what singles out Kullback--Leibler.
	
	\begin{definition}\label{def:fdiv}
		A function $f : \R^*_+ \to \R$ defines an \emph{$f$-divergence} if it is strictly
		convex on $\R^*_+$ and $f(1) = 0$. The associated divergence is
		\[
		\Df(\nu, \mu) =
		\begin{cases}
			\int_{\X} f\!\Bigl(\tfrac{\dd\nu}{\dd\mu}\Bigr) \dd\mu, & \text{if } \nu \ll \mu, \\[4pt]
			+\infty, & \text{otherwise.}
		\end{cases}
		\]
	\end{definition}
	\begin{definition}\label{def:eqfdiv}
		We  will say that $f$ is the \emph{generator} of $\Df$, and that functions $f, g$ are \emph{equivalent}, written $f \sim g$, if they define the same $f$-divergence up to a multiplicative constant, that is if $\Df$ is proportional to $D_g$.
	\end{definition}
	
	\begin{remark}\label{rem:equiv}
		Adding an affine term leads to an equivalent generator: for $A>0$ and $B \in \R$, if $g(x) = Af(x) + B(x-1)$ then $g \sim f$.
	\end{remark}
	
	\paragraph{The $f$-exponential family.}
	For a $C^2$ generator $f$, strict convexity on $\R_+^*$ makes $f'$ a strictly increasing
	homeomorphism from $\R_+^*$ onto an open interval
	\(
	f'(\R_+^*) \;\subseteq\; \R
	\) (a $C^1$-diffeomorphism wherever $f''$ does not vanish). The same Lagrange computation as
	in~\eqref{eq:maxent} shows that the $\Df$-constrained minimizer has density
	$(f')^{-1}(\theta^{\top} \mathcal{T})$ with respect to $\lambda$, provided $\theta$ is such
	that $\theta^{\top}\mathcal{T}(x) \in f'(\R_+^*)$ for $\lambda$-a.e.\ $x \in \X$. Thus $\theta$ must belong to the set
	\[
	\Theta_f \;:=\; \bigl\{\, \theta \in \R^{k+1} : \theta^{\top}\mathcal{T}(x) \in f'(\R_+^*) \ \text{for } \lambda\text{-a.e.\ } x \in \X \,\bigr\},
	\]
	which is convex since $f'(\R_+^*)$ is an interval, and on which $(f')^{-1}(\theta^{\top}\mathcal{T})$
	is a well-defined, $\lambda$-a.e.\ positive function. The $f$-exponential family is then
	\[
	\Ef \;=\; \Bigl\{\, \nu_\theta \in \Mset : \tfrac{\dd\nu_\theta}{\dd\lambda} = (f')^{-1}(\theta^{\top} \mathcal{T}),\ \theta \in \Theta_f,\ \textstyle\int_{\X} (f')^{-1}(\theta^{\top}\mathcal{T})\,\dd\lambda = 1 \,\Bigr\},
	\]
	obtained from $\Ekl$ by replacing $\exp = (f'_{\KL})^{-1}$ with $(f')^{-1}$, and $\R^{k+1}$ with
	$\Theta_f$; the integrability condition $\int_\X (f')^{-1}(\theta^{\top}\mathcal{T})\,\dd\lambda < \infty$
	for $\theta \in \Theta_f$ is implicit in requiring $\nu_\theta$ to be a (normalizable) probability
	measure. In particular, taking $f(x) = x\ln x - x + 1$, so that $f'(x) = \ln x$, $f'(\R_+^*) = \R$,
	$\Theta_f = \R^{k+1}$, and $(f')^{-1} = \exp$, recovers $\Ef = \Ekl$ with no restriction on $\theta$,
	a feature of the KL generator, not shared by all $f$-divergences for which $f'(\R_+^*) \subsetneq \R$
	(e.g. reverse KL or Hellinger, cf.\ Figure~\ref{fig:counterexamples}).
	
	\paragraph{What survives.}
	The maximum-entropy ``geodesic'' survives unchanged. For \emph{any} $f$-divergence, the
	minimizer of $\Df(\cdot,\lambda)$ over the moment set is still the unique element of $\Ef$ that matches
	the moments: whenever it exists, there is a unique $\nu_{\bar\theta} \in \Ef \cap \MT{\nu_0}$, and it minimizes
	$\nu \mapsto \Df(\nu, \lambda)$ over $\MT{\nu_0}$. The construction of the maximum-entropy
	reconstruction, the family $\Ef$ on which it lives, and the reduction to a finite-dimensional convex optimization program are not specific to KL.
	
	\paragraph{What breaks down.}
	However, the projection identity~\eqref{eq:pythagore}--\eqref{eq:bound} no longer holds: in general, nothing guarantees for
	an arbitrary $f$ and an arbitrary free family $\mathcal{T}$ that the $\Df$-projection of $\nu_0$ onto
	$\Ef$ coincides with the moment matcher $\nu_{\bar\theta}$. We give examples of such failures in
	Figure~\ref{fig:counterexamples}.
	
	\begin{figure}[H]
		\centering
		\begin{minipage}{0.48\textwidth}
			\centering
			\includegraphics[width=\linewidth]{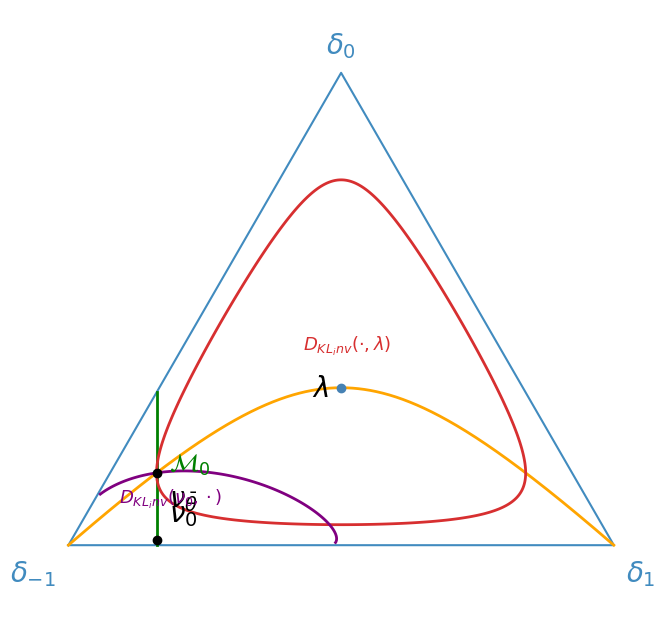}
			\par\smallskip
			\textit{Reverse KL divergence.}
		\end{minipage}
		\hfill
		\begin{minipage}{0.48\textwidth}
			\centering
			\includegraphics[width=\linewidth]{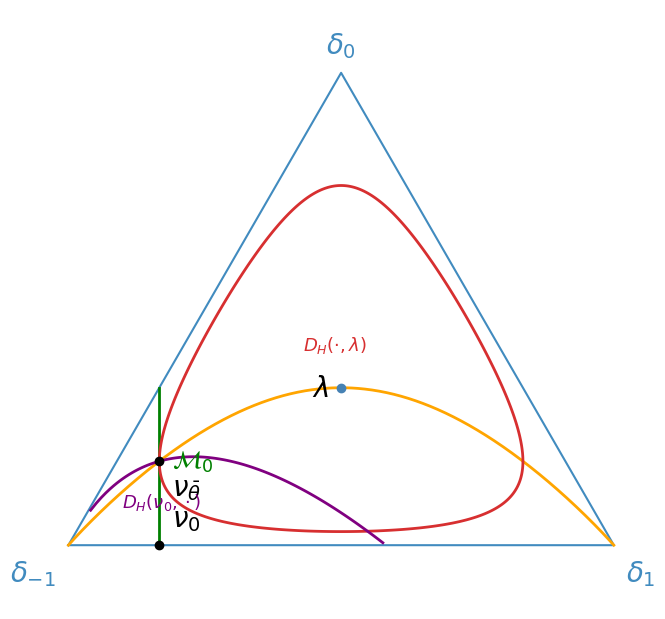}
			\par\smallskip
			\textit{Hellinger divergence.}
		\end{minipage}
		\caption{Two counterexamples on $\X = \{-1,0,1\}$ with the free family $(1,x)$, same colour code
			as Figure~\ref{KL}. The reverse KL divergence corresponds to $f(t) = -\ln t$ and the Hellinger
			divergence to $f(t) = \tfrac12(\sqrt{t}-1)^2$. In both cases the $\Df$-projection of $\nu_0$
			onto $\Ef$ no longer coincides with $\nu_{\bar\theta}$.}
		\label{fig:counterexamples}
	\end{figure}
	But are there other pairs $(f, \mathcal{T})$ for which the projection property still holds? Ideally one would
	hope for another $f$-divergence as powerful as KL (one enjoying the property for \emph{every}
	free family) or, failing that, one would at least like to identify the pairs $(f, \mathcal{T})$ for
	which it holds. Determining these pairs is the object of this paper, and we now name the property.
	
	\begin{definition}[$\mathcal{T}$-entropy-adapted]\label{def:adapted}
		Let $f$ be a $C^2$ $f$-divergence and $\mathcal{T} = (T^{0}, \dots, T^{k})$ a free family with $T^{0} \equiv 1$
		and $k \geq 1$. We say that $f$ is \emph{$\mathcal{T}$-entropy-adapted} if, for every finite space $\X = \{x_1,\dots,x_n\} \subset \R$
		with reference probability measure $\lambda$ of full support and every $\nu \in \PX$ equivalent to $\lambda$, if $\Ef \cap \MT{\nu}$ is non-empty its (necessarily unique) element $\nu_{\bar\theta}$ is such that
		\[
		\Df(\mu, \nu_\theta) \;\ge\; \Df(\mu, \nu_{\bar\theta})
		\qquad \text{for all } \mu \in \MT{\nu} \text{ and all } \nu_\theta \in \Ef\,.
		\]
	\end{definition}
	
	In words: $f$ is $\mathcal{T}$-entropy-adapted when, for any measure $\mu$ sharing the moments of $\nu$,
	the $\Df$-projection of $\mu$ onto $\Ef$ is always the single moment-matching point
	$\nu_{\bar\theta}$. This is exactly the property that made KL so usable (that the projection
	onto the family is computed by moment matching) now demanded uniformly over the whole moment
	set. Figure~\ref{fig:counterexamples} shows that the reverse KL and Hellinger divergences are
	\emph{not} $(1,x)$-entropy-adapted.
	
	\begin{remark}\label{rem:finite-spaces}
		We restrict Definition~\ref{def:adapted} to finite spaces for two reasons. First, the proof of Theorem~\ref{thm:main} only ever uses the hypothesis on finite spaces (Remark~\ref{rem:finite}), so a weaker hypothesis makes the ``only if'' direction stronger. Second, on infinite spaces the moment matcher may fail to exist even for KL: take $\X = [0,\infty)$, $T(x) = x$ and $\lambda(\dd x) \propto e^{-x}(1+x^3)^{-1}\dd x$. Then $e^{\theta x}$ is $\lambda$-integrable only for $\theta \leq 1$, so the means of the elements of $\Ekl$ are bounded by some $m^* < \infty$, while measures $\nu \sim \lambda$ with finite mean larger than $m^*$ exist. A definition quantifying over all spaces would therefore exclude KL itself. Finally, on a finite space all integrals are finite sums, and $\Ef$ is simply parametrized by $\{\theta \in \Theta_f : \sum_i \lambda_i (f')^{-1}(\theta^\top \mathcal{T}(x_i)) = 1\}$. The definition could nevertheless be extended to
		infinite spaces under additional assumptions guaranteeing the existence of the moment matcher, for
		instance bounded moment functions $\mathcal{T}$, or more generally measures $\nu$ whose moment vector
		lies in the image of the moment map of $\Ef$; Theorem~\ref{thm:main} would remain
		valid, since finite spaces are still included.
	\end{remark}
	
	Our main result is that, whatever the free family, KL is the
	\emph{only} smooth $f$-divergence with this property up to the harmless equivalence of Definition~\ref{def:eqfdiv}.
	
	\begin{theorem}\label{thm:main}
		Let $\mathcal{T} = (T^{0}, \dots, T^{k})$ be a free family with $T^{0} \equiv 1$ and $k \geq 1$, and let $f$ be
		a $C^2$ $f$-divergence. Then, we have:
		\[
		\text{$f$ is $\mathcal{T}$-entropy-adapted iff up to equivalence, }  \Df = \KL\,.
		\]
	\end{theorem}
	
	\section{Connections and implications of this result}
	\label{s:Connections}
	\subsection{Moment-based Distributional Reinforcement Learning}
	
	An application of this projection property appears in DistRL. In contrast with classical
	reinforcement learning, which focuses primarily on the expected return, DistRL aims at estimating the full distribution of the random return. Since
	this distribution is an infinite-dimensional object, practical algorithms
	must necessarily rely on finite-dimensional representations.
	
	A recent approach proposed in \citet{bringas2026momentmatching} represents
	return distributions through a finite collection of moments rather than
	through a discrete approximation of the full distribution. The motivation is
	that suitably chosen statistics, in particular polynomial moments, are
	\emph{Bellman-closed}: they can be propagated exactly through the
	distributional Bellman recursion. This avoids the repeated projection steps
	used by standard distributional reinforcement learning algorithms, and
	therefore prevents approximation errors from accumulating along the Bellman
	iterations. Once the moments of the return distribution have been computed,
	a distribution is reconstructed only once, at the end of the procedure.
	
	More precisely, let $\nu$ denote the return distribution and let
	$T=(T^1,\ldots,T^k)$ be a family of polynomial features, so that $\mathcal{T} = (1, T)$ in the notation of Section~\ref{s:Preliminaries}. The Bellman
	recursion provides the moment vector
	\[
	\nu(T)
	=
	\left(
	\int T^1\,\dd\nu,\ldots,\int T^k\,\dd\nu
	\right),
	\]
	and the final reconstruction is selected among all probability measures
	satisfying these constraints. Under appropriate regularity assumptions, it is possible to obtain an approximation guarantee that does not depend on
	the particular reconstruction rule. Any distribution $\widehat{\nu}$ that
	matches the first $k$ polynomial moments satisfies a Wasserstein bound of
	order
	\[
	W_1(\nu,\widehat{\nu})=O(k^{-1}).
	\]
	This illustrates a first fundamental advantage of moment representations:
	the matching of sufficiently many moments alone already guarantees weak
	approximation of the target distribution.
	
	However, substantially stronger guarantees are obtained when the
	reconstruction is chosen according to the maximum-entropy principle.
	Given the computed moments, the reconstruction is then
	\[
	\nu^\star
	\in
	\argmin_{\mu \in \MT{\nu}}
	\KL(\mu,\lambda),
	\]
	where $\lambda$ is a reference measure. As usual, $\nu^\star$ belongs to
	the exponential family $\Ekl$ generated by the same statistics, which we parametrize here as
	\[
	\Ekl
	=
	\left\{
	\nu_\theta:
	\frac{\dd\nu_\theta}{\dd\lambda}
	=
	\exp\!\left(
	\theta^\top T-A(\theta)
	\right)
	\right\},
	\qquad
	A(\theta) = \ln \int_{\X} \exp(\theta^\top T)\,\dd\lambda\,,
	\]
	where $A$ is the log-partition function, so that $\theta_0 = -A(\theta)$ in the parametrization
	of Section~\ref{s:Preliminaries}.
	A crucial step in the analysis of \citet{bringas2026momentmatching} is that
	this maximum-entropy reconstruction is simultaneously the information
	projection of the true return distribution onto $\Ekl$:
	\[
	\nu^\star
	=
	\argmin_{\nu_\theta\in\Ekl}
	\KL(\nu,\nu_\theta).
	\]
	Equivalently, for every $\nu_\theta\in\Ekl$,
	\[
	\KL(\nu,\nu_\theta)
	=
	\KL(\nu,\nu^\star)
	+
	\KL(\nu^\star,\nu_\theta).
	\]
	This identity makes it possible to compare the reconstruction error with any
	convenient element of the exponential family. In particular, one can choose
	$\nu_\theta$ such that $\ln(\dd\nu_\theta/\dd\lambda)$ approximates the
	log-density of the target distribution. The problem is thereby reduced to a
	classical approximation problem for $\ln(\dd\nu/\dd\lambda)$, leading to fast
	KL convergence rates under suitable regularity assumptions.
	
	Theorem~\ref{thm:main} clarifies why the Kullback--Leibler divergence plays a
	distinguished role in this argument. One could consider replacing
	maximum entropy by a generalized entropy reconstruction associated with
	another $f$-divergence,
	\[
	\nu_f^\star
	\in
	\argmin_{\mu \in \MT{\nu}}
	\Df(\mu,\lambda).
	\]
	The resulting reconstruction belongs to the corresponding generalized
	exponential family $\Ef$ of Section~\ref{s:Preliminaries}.
	Moment propagation through the Bellman recursion would remain unchanged:
	the same polynomial moments could still be computed exactly, and the
	reconstruction $\nu_f^\star$ would still match them. What is lost in
	general is the second interpretation of this reconstruction as the best
	$\Df$-approximation of the true return distribution within $\Ef$.
	
	Indeed, Theorem~\ref{thm:main} shows that, among $f$-divergences, KL is the only one for which this
	compatibility holds for every return distribution, and this whatever the family of statistics that is
	propagated, polynomial or not.
	Consequently, the Pythagorean argument used in
	\citet{bringas2026momentmatching} to transform a moment reconstruction
	problem into an approximation problem for the log-density is not a generic
	property of generalized entropy methods. Within the class of
	$f$-divergences, it is a specific feature of KL.
	
	This distinction also separates two levels of guarantees in
	moment-based DistRL. The Wasserstein convergence resulting solely from
	moment matching is independent of the entropy criterion and therefore
	extends to arbitrary reconstruction procedures. In contrast, the refined
	information-theoretic guarantees based on projection onto the associated
	finite-dimensional family exploit precisely the compatibility characterized
	by Theorem~\ref{thm:main}. Our result therefore provides a structural
	explanation for the special effectiveness of maximum entropy and KL in this
	moment-based reconstruction framework.
	
	\subsection{Exponential families and statistical inference}
	
	The projection property characterized in Theorem~\ref{thm:main} also has a
	natural interpretation in statistics, where exponential families play a
	central role. Consider a classical exponential family with density
	\[
	p_\theta(x)
	=
	\exp\!\left(
	\theta^\top T(x)-A(\theta)
	\right)
	\]
	with respect to a reference measure $\lambda$, where $T=(T^1,\ldots,T^k)$
	is the vector of canonical statistics and $A$ is the log-partition
	function. Given an arbitrary probability measure $\nu_0$, approximating
	$\nu_0$ within this family by minimizing the Kullback--Leibler divergence
	amounts to solving
	\[
	\min_{\theta}
	\KL(\nu_0,\nu_\theta).
	\]
	Up to a term independent of $\theta$, the objective can be written as
	\[
	\KL(\nu_0,\nu_\theta)
	=
	A(\theta)-\theta^\top \nu_0(T)+C(\nu_0).
	\]
	Since
	\[
	\nabla A(\theta)=\nu_\theta(T),
	\]
	the first-order optimality condition is precisely
	\[
	\nu_\theta(T)=\nu_0(T).
	\]
	Hence, the KL-best approximation of an arbitrary distribution by an
	exponential family is characterized entirely by the matching of the
	canonical moments.
	
	This identity is closely related to one of the basic properties of
	exponential families in statistical inference. The part of $\KL(\nu_0,\nu_\theta)$ that depends on $\theta$
	is the cross-entropy $-\nu_0(\ln p_\theta) = A(\theta)-\theta^\top \nu_0(T)$, which makes sense for every
	probability measure $\nu_0$, even one that is not absolutely continuous with respect to $\lambda$. If $\nu_0$ is replaced by
	the empirical measure
	\[
	\widehat{\nu}_n
	=
	\frac1n\sum_{i=1}^n \delta_{X_i},
	\]
	then minimizing this cross-entropy,
	\[
	-\widehat{\nu}_n(\ln p_\theta)
	=
	-\frac1n\sum_{i=1}^n \ln p_\theta(X_i),
	\]
	amounts to maximizing the log-likelihood of the observations. The likelihood equations therefore
	take the form
	\[
	\nu_\theta(T)
	=
	\widehat{\nu}_n(T)
	=
	\frac1n\sum_{i=1}^n T(X_i).
	\]
	Thus, in a classical exponential family, maximum likelihood, KL projection,
	and moment matching lead to the same system of equations. In particular,
	although the unknown distribution may be infinite-dimensional, its best
	approximation in the model can be determined from the finite-dimensional
	vector of expectations of the canonical statistics.
	
	The $f$-exponential families considered in this paper provide a natural
	generalization of this construction. For a generator $f$, the corresponding
	maximum-entropy family is characterized by
	\[
	f'(p_\theta)=\theta_0 + \theta^\top T,
	\qquad\text{or equivalently}\qquad
	p_\theta=(f')^{-1}(\theta_0 + \theta^\top T),
	\]
	where $\theta_0$ is fixed by the normalization constraint. One may therefore ask whether the same
	statistical simplification survives when KL is replaced by an arbitrary
	$f$-divergence: is the $\Df$-best approximation of $\nu_0$ within $\Ef$ still determined solely by the ordinary moment equations
	\(
	\nu_\theta(T)=\nu_0(T)?
	\)
	
	Theorem~\ref{thm:main} shows that, if this property is required universally,
	the answer is negative for every $f$-divergence except KL. More precisely,
	among $f$-divergences, KL is the unique one, up to equivalence, for which
	the best approximation in the family generated by the same function $f$
	can always be recovered by matching the original linear moments. The
	classical relation between the best approximation by a divergence and the moment equations 
	is therefore not a generic feature of generalized exponential families,
	but a property of the KL/exponential-family geometry.
	
	\section{Proof of Theorem~\ref{thm:main}}
	\label{s:Proof}
	
	The ``if'' direction of the theorem (KL is $\mathcal{T}$-entropy-adapted) is classical~\citep{csiszar1975divergence}: on a finite space the moment matcher exists (Section~\ref{s:Preliminaries}), and the Pythagorean identity of Lemma~\ref{lem:pyth} gives the inequality.
	We prove here that KL is (up to a multiplicative constant) the only $f$-divergence for which $f$ is
	$\mathcal{T}$-entropy-adapted.
	
	The strategy is as follows: for any generator $f$, define $h(x) \eqdef f(x) - x f'(x)$. We first show (Lemma~\ref{lem:const}) that if $h'$ is constant then $f$ is equivalent to a KL generator (the converse is immediate, see Remark~\ref{rem:KLcheck}). The whole proof then consists in extracting this constancy from the entropy-adapted
	hypothesis. The mechanism is a Lagrange computation anchored at a point $\nu_\ast$ of $\Ef$: the
	first-order optimality of the $\Df$-projection along the moment set of $\nu_\ast$ yields a linear
	identity in $h$ (Step~\ref{step:lagrange}), which an orthogonality argument upgrades to an affine
	relation between $h'$ and the moment functions (Step~\ref{step:affine}); testing this relation on
	well-chosen finite configurations then forces $h'$ to be constant. When $f''(1) \neq 0$, the anchor
	can be taken to be $\lambda$ itself and every ingredient is elementary. The real difficulty of the
	proof is the possible degeneracy $f''(1) = 0$, which makes $\theta \mapsto p_\theta$
	non-differentiable at $\theta = 0$ (i.e.\ at $\lambda$) and forces us to move the anchor, via a Sard-type argument, to a
	non-degenerate point of $\Ef$ (Lemma~\ref{lem:anchor}).
	
	Throughout this section we write $p_\theta = \dd\nu_\theta/\dd\lambda = (f')^{-1}(\theta^{\top} \mathcal{T})$
	for the density of $\nu_\theta \in \Ef$, and $T(x) \eqdef \bigl(T^1(x), \dots, T^k(x)\bigr) \in \R^k$
	for the vector of non-constant moment functions.
	
	\begin{lemma}\label{lem:const}
		Let $f$ be a $C^2$ $f$-divergence and set $h(x) = f(x) - x f'(x)$, so that $h'(x) = -x f''(x)$.
		If $h'$ is constant on $\R^*_+$ (equivalently, if $x \mapsto x f''(x)$ is constant) then,
		up to equivalence, $\Df = \KL$.
	\end{lemma}
	
	\begin{proof}
		Write $h'(x) = -x f''(x) = -A$ for all $x > 0$. If $A = 0$ then $f'' \equiv 0$ and $f$ is affine,
		contradicting strict convexity; convexity of $f$ then forces $A > 0$. Integrating
		$f''(x) = A/x$ twice, there exist $C, D \in \R$ such that
		\[
		f(x) = A\, x \ln x - A x + C x + D\,,
		\]
		and the normalization $f(1) = 0$ gives $D = A - C$, that is
		\[
		f(x) = A\, x \ln x + (C - A)(x - 1)\,.
		\]
		By Remark~\ref{rem:equiv}, $f \sim x \mapsto x \ln x$, whose divergence is exactly $\KL$.
	\end{proof}
	
	We begin with two reductions.
	
	\begin{remark}\label{rem:finite}
		By Definition~\ref{def:adapted}, the hypothesis only concerns finite spaces
		$\X = \{x_1, \dots, x_n\} \subset \R$ (distinct points). Reference measures are then vectors: we
		write $\lambda = \sum_{i=1}^n \lambda_i \delta_{x_i}$ with $\lambda_i > 0$ and
		$\sum_i \lambda_i = 1$, and $T^j_i \eqdef T^j(x_i)$, thereby identifying functions on $\X$ with
		vectors of $\R^n$. We set
		\[
		S_\X \;\eqdef\; \spn\bigl(1, T^1|_\X, \dots, T^k|_\X\bigr) \;\subset\; \R^n\,.
		\]
		Since $T^0 \equiv 1$, a probability measure $\mu$ with full support on $\X$ belongs to $\MT{\nu}$
		if and only if $\sum_i \mu_i g_i = \sum_i \nu_i g_i$ for every $g \in S_\X$: the moment sets depend
		on $\mathcal{T}$ only through $S_\X$, and the family $\mathcal{T}$ need not be free on $\X$. Every
		integral is a finite sum, so all the differentiations under the integral sign performed below are
		automatic.
	\end{remark}
	
	\begin{remark}\label{rem:translation}
		Let $c = (c_1, \dots, c_k) \in \R^k$. Translating each $T^j$ by the constant $c_j$ changes neither
		the moment sets (all measures involved being probability measures) nor the family $\Ef$, since
		$\spn(1, T^1 - c_1, \dots, T^k - c_k) = \spn(1, T^1, \dots, T^k)$. Hence, if $f$ is
		$\mathcal{T}$-entropy-adapted, it is also entropy-adapted for the translated family, for
		\emph{every} $c \in \R^k$. We will use this freedom with $c = \tau$, a barycentre attached to the
		configurations constructed below.
	\end{remark}
	
	\paragraph{Notation.}
	From now on we normalize $f'(1) = 0$, which by Remark~\ref{rem:equiv} changes neither $\Df$, nor
	$\Ef$, nor $f''$. We write
	\[
	Z \eqdef \{x > 0 : f''(x) = 0\}\,.
	\]
	By strict convexity, $Z$ is closed with empty interior (on any interval where $f'' \equiv 0$ the
	function $f$ would be affine). The derivative $f'$ is continuous and strictly increasing, hence a
	homeomorphism from $\R^*_+$ onto the \emph{open} interval $I \eqdef f'(\R^*_+)$, which contains
	$0 = f'(1)$; consequently $f'(Z)$ is closed in $I$, and the inverse
	$(f')^{-1} : I \to \R^*_+$ is continuous, increasing, and $C^1$ on the open set
	$I \setminus f'(Z)$, with
	\[
	\bigl((f')^{-1}\bigr)' \;=\; \frac{1}{f'' \circ (f')^{-1}}
	\]
	there, by the inverse function theorem. Finally, by Sard's theorem in dimension
	one\footnote{Elementary in dimension one: fix a compact interval $[a,b] \subset \R^*_+$ and
		$\varepsilon > 0$. By uniform continuity of $f''$ on $[a,b]$, there is $\delta > 0$ such that
		$|f''| \le \varepsilon$ at every point of $[a,b]$ at distance $\le \delta$ from $Z$ (recall
		$f'' = 0$ on $Z$). Split $[a,b]$ into at most $(b-a)/\delta + 1$ intervals of length $\le \delta$:
		each one meeting $Z$ is made of such points, so $f'$ has oscillation $\le \varepsilon\delta$ on
		it, and $f'(Z \cap [a,b])$ has outer measure at most $\varepsilon(b - a + \delta)$. Let
		$\varepsilon \to 0$, then exhaust $\R^*_+$ by compact intervals.}, the set $f'(Z)$ of critical
	values of $f'$ is Lebesgue-negligible. Note right away that when $f''(1) \neq 0$ none of this
	machinery will be needed: it only serves to handle the degenerate case $f''(1) = 0$.
	
	\paragraph{Anchored configurations.}
	By an \emph{anchored configuration} we mean the data of a finite space
	$\X = \{x_1, \dots, x_n\}$, of weights $\lambda_i > 0$ with $\sum_i \lambda_i = 1$, and of a
	parameter $\theta^\ast \in \R^{k+1}$ such that, setting
	$m_i \eqdef \theta^\ast_0 + \sum_{j=1}^k \theta^\ast_j T^j_i$ and $p_i \eqdef (f')^{-1}(m_i)$,
	\[
	m_i \in I\,, \qquad f''(p_i) > 0\,, \qquad \text{and} \qquad \sum_{i=1}^n \lambda_i\, p_i = 1\,.
	\]
	The first two conditions mean $m_i \in I \setminus f'(Z)$. The last one says that
	$p_\ast \eqdef (p_1, \dots, p_n)$ is a probability density with respect to $\lambda$; since
	$p_\ast = (f')^{-1}(\theta^{\ast\top} \mathcal{T})$, the measure $\nu_\ast \eqdef p_\ast \lambda$ then
	belongs to $\Ef$, and we write $\nu^\ast_i \eqdef \lambda_i\, p_i > 0$ for its weights. The following
	lemma provides anchored configurations on any prescribed finite set of points; its Sard argument is
	only needed when $f''(1) = 0$, as Remark~\ref{rem:trivialanchor} below explains.
	
	\begin{lemma}[Existence of a non-degenerate anchor]\label{lem:anchor}
		Let $y_1, \dots, y_N \in \R$ be distinct points, $j \in \{1, \dots, k\}$, and fix two indices
		$\kappa, \Lambda$ with $T^j(y_\kappa) < T^j(y_\Lambda)$. There exist $\theta^\ast_0 \in \R$ and
		$\theta^\ast_j > 0$ such that, setting all the other coordinates of $\theta^\ast$ to zero,
		$m_i \eqdef \theta^\ast_0 + \theta^\ast_j T^j(y_i)$ and $p_i \eqdef (f')^{-1}(m_i)$:
		\begin{itemize}
			\item[(i)] $m_i \in I \setminus f'(Z)$ for every $i \in \{1, \dots, N\}$; in particular
			$f''(p_i) > 0$, and $\theta \mapsto p_{\theta,i} \eqdef (f')^{-1}\bigl(\theta_0 + \sum_l \theta_l T^l(y_i)\bigr)$
			is $C^1$ in a neighbourhood of $\theta^\ast$;
			\item[(ii)] $p_\kappa < 1 < p_\Lambda$.
		\end{itemize}
		Moreover, let $\X$ be a subset of $\{y_1, \dots, y_N\}$ containing $y_\kappa$ and $y_\Lambda$, and
		split its other indices into two disjoint sets $A$ and $B$. There exists $\delta > 0$ such that, for
		all $(\epsilon, \eta) \in [0, \delta]^2$, the solution $(\lambda_\kappa, \lambda_\Lambda)$ of
		\[
		\lambda_\kappa + \lambda_\Lambda = 1 - \epsilon\, |A| - \eta\, |B|\,, \qquad
		\lambda_\kappa\, p_\kappa + \lambda_\Lambda\, p_\Lambda = 1 - \epsilon \sum_{i \in A} p_i - \eta \sum_{i \in B} p_i
		\]
		has positive coordinates and depends continuously on $(\epsilon, \eta)$. In particular, for
		$(\epsilon, \eta) \in (0, \delta]^2$, the weights $\lambda_i = \epsilon$ ($i \in A$),
		$\lambda_i = \eta$ ($i \in B$), $\lambda_\kappa$, $\lambda_\Lambda$, together with $\theta^\ast$,
		turn $\X$ into an anchored configuration.
	\end{lemma}
	
	\begin{proof}
		Write $\theta_0$, $\theta_1$ for $\theta^\ast_0$, $\theta^\ast_j$, and $T_i$ for $T^j(y_i)$. Fix
		$\theta_1 > 0$, to be shrunk below. Since $(f')^{-1}$ is increasing with $(f')^{-1}(0) = 1$,
		condition (ii) reads $m_\kappa < 0 < m_\Lambda$, i.e.
		\[
		\theta_0 \in J \eqdef \bigl(-\theta_1 T_\Lambda,\, -\theta_1 T_\kappa\bigr)\,,
		\]
		a non-empty open interval, of length $\theta_1 (T_\Lambda - T_\kappa)$. For $\theta_0 \in J$ one has
		$|\theta_0| \le \theta_1 \max(|T_\kappa|, |T_\Lambda|)$, hence
		\[
		\max_i |m_i| \;\le\; \theta_1 \Bigl( \max(|T_\kappa|, |T_\Lambda|) + \max_i |T_i| \Bigr)\,;
		\]
		since $I$ is open and contains $0$, taking $\theta_1$ small enough guarantees $m_i \in I$ for
		every $i$ and every $\theta_0 \in J$. With $\theta_1$ now fixed, the failure of the first part of
		(i) means $\theta_0 \in \bigcup_i \bigl(f'(Z) - \theta_1 T_i\bigr)$, a finite union of
		Lebesgue-null sets, hence a null set; since $J$ has positive measure, we may pick $\theta_0 \in J$
		outside it, so that $m_i \notin f'(Z)$ for every $i$. This proves (i)--(ii): the strict positivity
		$f''(p_i) > 0$ follows from $f'' \ge 0$ (convexity) together with $p_i \notin Z$, and the $C^1$
		regularity of $\theta \mapsto p_{\theta,i}$ near $\theta^\ast$ from the inverse function theorem at
		the non-critical point $m_i$.
		
		For the weights, the linear system in $(\lambda_\kappa, \lambda_\Lambda)$ has determinant
		$p_\Lambda - p_\kappa > 0$, independent of $(\epsilon, \eta)$, so its solution is continuous (indeed
		affine) in $(\epsilon, \eta)$. At $(\epsilon, \eta) = (0, 0)$ it equals
		\[
		(\lambda_\kappa, \lambda_\Lambda) \;=\;
		\Bigl( \frac{p_\Lambda - 1}{p_\Lambda - p_\kappa},\; \frac{1 - p_\kappa}{p_\Lambda - p_\kappa} \Bigr)\,,
		\]
		whose coordinates are positive by (ii); by continuity, they remain positive on $[0, \delta]^2$ for
		some $\delta > 0$. The two equations are exactly $\sum_i \lambda_i = 1$ and
		$\sum_i \lambda_i\, p_i = 1$, which, together with (i), is the definition of an anchored
		configuration.
	\end{proof}
	
	\begin{remark}\label{rem:trivialanchor}
		When $f''(1) \neq 0$, no lemma is needed: the choice $\theta^\ast = 0$ turns \emph{any} finite
		configuration with \emph{any} positive weights into an anchored configuration. Indeed $m_i = 0$
		for every $i$, so $p_i = (f')^{-1}(0) = 1$, $f''(p_i) = f''(1) \neq 0$, and the normalization
		$\sum_i \lambda_i p_i = \sum_i \lambda_i = 1$ is automatic; the anchor is then simply
		$\nu_\ast = \lambda$ itself, with $\nu^\ast_i = \lambda_i$.
	\end{remark}
	
	\paragraph{The weights $w$, the barycentre $\tau$, and the sets $\mathcal{A}$, $\mathcal{R}$,
		$\mathcal{V}$.}
	Given an anchored configuration, define, for $i \in \{1, \dots, n\}$ and $j \in \{1, \dots, k\}$,
	\[
	w_i \eqdef -\frac{1}{h'(p_i)} = \frac{1}{p_i\, f''(p_i)} > 0\,, \qquad
	W \eqdef \sum_{i=1}^n \nu^\ast_i\, w_i\,, \qquad
	\tau_j \eqdef \frac{1}{W} \sum_{i=1}^n \nu^\ast_i\, w_i\, T^j_i\,,
	\]
	and $\tau \eqdef (\tau_1, \dots, \tau_k)$. Since all the weights $\nu^\ast_i w_i$ are positive, each
	$\tau_j$ is a convex combination of the $T^j_i$ with positive weights; in particular
	$\min_i T^j_i < \tau_j < \max_i T^j_i$ whenever $T^j$ is non-constant on $\X$. In the trivial anchor
	of Remark~\ref{rem:trivialanchor} ($f''(1) \neq 0$, $\theta^\ast = 0$), these objects are
	transparent: $w_i \equiv 1/f''(1)$ is constant and $\tau_j = \int_{\X} T^j \dd\lambda$ is nothing but
	the $\lambda$-mean of $T^j$.
	
	Whenever an anchored configuration is fixed, we tacitly replace each $T^j$ by $T^j - \tau_j$, which
	is legitimate by Remark~\ref{rem:translation}; note that replacing simultaneously $\theta^\ast_0$ by
	$\theta^\ast_0 + \sum_j \theta^\ast_j \tau_j$ leaves every $m_i$ (hence every $p_i$, $w_i$, and
	$\nu_\ast$) unchanged, while the new barycentre vanishes. \emph{We may and do therefore assume that
		$\tau = 0$.} Finally, define
	\[
	\mathcal{A}(\nu_\ast) = \Bigl\{ r \in \R^n : \textstyle\sum_i \nu^\ast_i r_i g_i = \sum_i \nu^\ast_i g_i
	\ \text{for all } g \in S_\X \Bigr\},
	\qquad
	\mathcal{R}(\nu_\ast) = \mathcal{A}(\nu_\ast) \cap (\R_+^*)^n,
	\]
	\[
	\mathcal{V}(\nu_\ast) = \Bigl\{ v \in \R^n : \textstyle\sum_i \nu^\ast_i v_i g_i = 0
	\ \text{for all } g \in S_\X \Bigr\},
	\]
	so that $r \in \mathcal{R}(\nu_\ast)$ exactly when $r\nu_\ast$ is a probability measure of
	$\MT{\nu_\ast}$. Then $\mathcal{A}(\nu_\ast)$ is an affine subspace with direction
	$\mathcal{V}(\nu_\ast)$; since $(1, \dots, 1) \in \mathcal{A}(\nu_\ast) \cap (\R_+^*)^n$ and
	$(\R_+^*)^n$ is open in $\R^n$, the set $\mathcal{R}(\nu_\ast)$ is a non-empty open subset of
	$\mathcal{A}(\nu_\ast)$. For the inner product $\langle u, z \rangle_{\nu^\ast} = \sum_i \nu^\ast_i u_i z_i$,
	one has $\mathcal{V}(\nu_\ast) = S_\X^{\perp_{\nu^\ast}}$.
	
	\begin{step}[Anchored Lagrange identity]\label{step:lagrange}
		For every anchored configuration, every $r \in \mathcal{R}(\nu_\ast)$ and every
		$j \in \{1, \dots, k\}$,
		\begin{equation}\label{eq:anchored}
			\sum_{i=1}^n \nu^\ast_i\, w_i\, T^j_i\, h(r_i) \;=\; 0\,.
		\end{equation}
	\end{step}
	
	\begin{proof}
		The measure $\nu_\ast$ belongs to $\Ef \cap \MT{\nu_\ast}$, so the uniqueness part of
		Definition~\ref{def:adapted}, applied with $\nu = \nu_\ast$, gives $\nu_{\bar\theta} = \nu_\ast$:
		the moment matcher of $\nu_\ast$ is $\nu_\ast$ itself. The adaptedness inequality then yields,
		for $\mu \eqdef r\nu_\ast \in \MT{\nu_\ast}$,
		\[
		\Df(\mu, \nu_\theta) \;\ge\; \Df(\mu, \nu_\ast) \qquad \text{for all } \nu_\theta \in \Ef\,.
		\]
		The density of $\mu$ with respect to $\lambda$ being $r p_\ast$, this says that
		$\theta^\ast$ minimizes
		\[
		\theta \;\longmapsto\; \sum_i \lambda_i\, f\!\Bigl(\frac{r_i\, p_i}{p_{\theta,i}}\Bigr)\, p_{\theta,i}
		\qquad \text{subject to} \qquad \sum_i \lambda_i\, p_{\theta,i} = 1\,,
		\]
		with $p_{\theta,i} = (f')^{-1}\bigl(\theta_0 + \sum_l \theta_l T^l_i\bigr)$, over a neighbourhood of
		$\theta^\ast$ in $\R^{k+1}$ on which every $\theta_0 + \sum_l \theta_l T^l_i$ stays in the open set
		$I \setminus f'(Z)$. Since $f''(p_i) > 0$ and $f''$ is continuous, both the objective and the
		constraint are $C^1$ near $\theta^\ast$, with
		\[
		\frac{\partial p_{\theta,i}}{\partial \theta_l} = \frac{T^l_i}{f''(p_{\theta,i})}
		\qquad (l = 0, \dots, k,\ T^0 \equiv 1)\,.
		\]
		Moreover, for any fixed $q > 0$,
		\[
		\frac{\dd}{\dd p}\Bigl[ f\!\Bigl(\frac{q}{p}\Bigr)\, p \Bigr]
		= f'\!\Bigl(\frac{q}{p}\Bigr)\Bigl(-\frac{q}{p^2}\Bigr)\, p + f\!\Bigl(\frac{q}{p}\Bigr)
		= f\!\Bigl(\frac{q}{p}\Bigr) - \frac{q}{p}\, f'\!\Bigl(\frac{q}{p}\Bigr)
		= h\!\Bigl(\frac{q}{p}\Bigr)\,.
		\]
		Applying this with $q = r_i p_i$ and evaluating at $\theta = \theta^\ast$, where
		$p_{\theta^\ast, i} = p_i$ (so that $q/p = r_i$), the partial derivatives in $\theta_l$ of the
		objective and of the constraint at $\theta^\ast$ are
		\[
		\sum_i \lambda_i\, T^l_i\, \frac{h(r_i)}{f''(p_i)}
		\qquad \text{and} \qquad
		\sum_i \lambda_i\, T^l_i\, \frac{1}{f''(p_i)}\,, \qquad l = 0, \dots, k\,.
		\]
		The constraint gradient is non-zero (its $\theta_0$-coordinate is
		$\sum_i \lambda_i / f''(p_i) > 0$) so the Lagrange multiplier rule applies: there exists
		$\alpha \in \R$ such that, for $l = 0, \dots, k$,
		\[
		\sum_i \lambda_i\, T^l_i\, \frac{h(r_i)}{f''(p_i)} = \alpha \sum_i \lambda_i\, T^l_i\, \frac{1}{f''(p_i)}\,.
		\]
		Since $\lambda_i / f''(p_i) = \lambda_i p_i \cdot \frac{1}{p_i f''(p_i)} = \nu^\ast_i w_i$, the
		equation with index $l = j \geq 1$ reads
		\[
		\sum_i \nu^\ast_i\, w_i\, T^j_i\, h(r_i)
		\;=\; \alpha \sum_i \nu^\ast_i\, w_i\, T^j_i
		\;=\; \alpha\, W \tau_j \;=\; 0\,,
		\]
		by the standing convention $\tau = 0$. (The equation with index $l = 0$ merely determines $\alpha$,
		which we do not need.)
	\end{proof}
	
	The next step is the core of the proof.
	
	\begin{step}[Anchored affine relation]\label{step:affine}
		For every anchored configuration, every $r \in \mathcal{R}(\nu_\ast)$ and every
		$j \in \{1, \dots, k\}$,
		\[
		\zeta^j(r) \;\eqdef\; \bigl(w_i\, T^j_i\, h'(r_i)\bigr)_{1 \le i \le n} \;\in\; S_\X\,,
		\]
		that is, there exist $a_r, b_{r,1}, \dots, b_{r,k} \in \R$ (depending on $j$) such that
		\[
		\forall i \in \{1, \dots, n\}, \qquad
		w_i\, T^j_i\, h'(r_i) \;=\; a_r + \sum_{l=1}^k b_{r,l}\, T^l_i\,.
		\]
	\end{step}
	
	\begin{proof}
		Let $r \in \mathcal{R}(\nu_\ast)$ and $v \in \mathcal{V}(\nu_\ast)$. Since $\mathcal{R}(\nu_\ast)$
		is open in $\mathcal{A}(\nu_\ast)$, there is $t_0 > 0$ such that $r + tv \in \mathcal{R}(\nu_\ast)$
		for all $t \in (-t_0, t_0)$, and identity~\eqref{eq:anchored} (whose coefficients
		$\nu^\ast_i w_i T^j_i$ do not depend on $r$) holds at every $r + tv$. Differentiating in $t$ at
		$t = 0$ (licit, $h$ being $C^1$ since $f$ is $C^2$),
		\[
		\sum_i \nu^\ast_i\, w_i\, T^j_i\, v_i\, h'(r_i) = 0
		\qquad \text{for all } v \in \mathcal{V}(\nu_\ast)\,.
		\]
		This says that $\zeta^j(r)$ is $\langle \cdot, \cdot \rangle_{\nu^\ast}$-orthogonal to
		$\mathcal{V}(\nu_\ast) = S_\X^{\perp_{\nu^\ast}}$, hence
		$\zeta^j(r) \in \bigl(S_\X^{\perp_{\nu^\ast}}\bigr)^{\perp_{\nu^\ast}} = S_\X$. This concludes.
	\end{proof}
	
	\begin{remark}\label{rem:KLcheck}
		As a sanity check, for the KL generator $f(x) = x \ln x - x + 1$ (which satisfies $f'(1) = 0$)
		one has $h' \equiv -1$ and $w_i = 1/(p_i \cdot \tfrac{1}{p_i}) = 1$, so that
		$\zeta^j(r) = -T^j|_\X \in S_\X$ whatever $r$: as it must (KL being
		$\mathcal{T}$-entropy-adapted) Step~\ref{step:affine} carries no information in that case. It is
		the possible non-constancy of $h'$ for other generators that the relation will now be played
		against. For $k = 1$, Step~\ref{step:affine} reads $w_i T_i h'(r_i) = a_r + b_r T_i$.
	\end{remark}
	
	This last step lets us conclude, by testing its conclusion on well-chosen configurations. It rests
	on the following elementary fact: \emph{if $P \subset \{1, \dots, n\}$ is such that the restriction
		map $S_\X \to \R^P$ is injective, then any element of $S_\X$ vanishing on $P$ is zero.} The next
	lemma provides such sets of points, including for pathological moment functions.
	
	\begin{lemma}[Choice of points]\label{lem:points}
		Let $F = \spn(1, T^1, \dots, T^k)$, viewed as a space of functions on $\R$, let $F_0 \subset F$ be
		the subspace of its finitely supported elements, and let $B$ be the union of the supports of the
		elements of $F_0$. Then $B$ is finite. Let $F'$ be the space of restrictions to $\R \setminus B$ of
		the elements of $F$, and $d \eqdef \dim F' \ge 1$. Then:
		\begin{itemize}
			\item[(a)] for every finite $C \subset \R \setminus B$, the restriction map from $F'$ to functions
			on $\R \setminus (B \cup C)$ is injective;
			\item[(b)] for every finite $C \subset \R \setminus B$, there exist $d$ points
			$y_1, \dots, y_d \in \R \setminus (B \cup C)$ such that $g \mapsto (g(y_1), \dots, g(y_d))$ is a
			bijection from $F'$ onto $\R^d$;
			\item[(c)] $d = 1$ if and only if every $T^j$ is constant on $\R \setminus B$.
		\end{itemize}
	\end{lemma}
	
	\begin{proof}
		$F_0$ is finite-dimensional; if $g^1, \dots, g^m$ is a basis, every element of $F_0$ is supported in
		$\bigcup_l \mathrm{supp}\, g^l$, which is finite, so $B$ is finite. Since $\R \setminus B$ is
		non-empty, the constant $1$ has a non-zero restriction, so $d \ge 1$.
		(a) Let $g \in F$ vanish on $\R \setminus (B \cup C)$. Then $g$ is supported in the finite set
		$B \cup C$, so $g \in F_0$ and $g$ vanishes outside $B$: its restriction to $\R \setminus B$ is zero.
		(b) By (a), the restrictions of the elements of $F'$ to $E \eqdef \R \setminus (B \cup C)$ form a
		space $F''$ of dimension $d$. The evaluations $g \mapsto g(y)$, $y \in E$, span the dual of $F''$:
		otherwise some non-zero $g \in F''$ would vanish at every point of $E$. Extract a basis
		$y_1, \dots, y_d$ of evaluations; the corresponding map is a bijection $F'' \to \R^d$, hence, by (a),
		a bijection $F' \to \R^d$.
		(c) $F'$ contains the constants, and is spanned by them and the restrictions of the $T^j$.
	\end{proof}
	
	\begin{proof}[Proof of Theorem~\ref{thm:main}]
		Fix $s, u > 0$; our goal is to prove that $h'(s) = h'(u)$ when $f''(1) \neq 0$
		(Lemma~\ref{lem:const} then concludes), and that $h'(u) = 0$ when $f''(1) = 0$, taking $s = 1$
		(which makes $f$ affine and shows that this case in fact never occurs). Both
		conclusions will be read off Step~\ref{step:affine} on a configuration built once and for all.
		
		\bigskip
		\noindent\textbf{A preliminary reduction.}
		Suppose first that some value $V \in \R^k$ of the vector $T$ is attained at least three times, and
		pick three distinct points $x_1, x_2, x_3 \in T^{-1}(V)$. Since $\mathcal{T}$ is free, $T$ is not
		constant on $\R$, so there is $x_4 \in T^{-1}(V')$ for some value $V' \neq V$; pick $j$ with
		$V'_j \neq V_j$. On $\X = \{x_1, \dots, x_4\}$, every $g \in S_\X$ takes the same value at $x_1$,
		$x_2$ and $x_3$. Lemma~\ref{lem:anchor} (with coordinate $j$, pivots $\{3, 4\}$ ordered according
		to the sign of $V'_j - V_j$, $A = \{1, 2\}$ and $B = \emptyset$), or Remark~\ref{rem:trivialanchor}
		with weights $(\epsilon, \epsilon, \frac12 - \epsilon, \frac12 - \epsilon)$ if $f''(1) \neq 0$,
		provides an anchored configuration on these four points in which the weights at $x_1$ and $x_2$ are
		small while $\nu^\ast_3$ stays bounded away from $0$. Setting then $r_1 = s$, $r_2 = u$, $r_4 = 1$ and
		\[
		r_3 = 1 - \frac{\nu^\ast_1\,(s-1) + \nu^\ast_2\,(u-1)}{\nu^\ast_3}\,,
		\]
		positive for that choice of weights, the constraints of $\mathcal{R}(\nu_\ast)$ hold by inspection
		(group the indices $1, 2, 3$, where every $g \in S_\X$ takes a common value). After the standing
		translation, $\tau_j = 0$ lies strictly between the two values $V_j$ and $V'_j$ taken by $T^j$ on
		the configuration, so $V_j \neq 0$; and Step~\ref{step:affine} at the indices $1$ and $2$, which
		share the value $V$ hence the weight $w_V$, and at which every element of $S_\X$ coincides, gives
		\[
		w_V\, V_j\, h'(s) \;=\; w_V\, V_j\, h'(u)\,,
		\]
		so $h'(s) = h'(u)$. As $s, u$ were arbitrary, $h'$ is constant and Lemma~\ref{lem:const} already
		proves the theorem in this situation (note that the argument used neither the value of
		$f''(1)$ nor $s = 1$). \emph{From now on we may therefore assume that every value of $T$ is
			attained at most twice.} With the notation of Lemma~\ref{lem:points}, this forces $d \ge 2$:
		otherwise, by Lemma~\ref{lem:points}(c), $T$ would be constant on the infinite set $\R \setminus B$.
		We fix $j$ such that $T^j$ is non-constant on $\R \setminus B$.
		
		\bigskip
		\noindent\textbf{The common configuration.}
		By Lemma~\ref{lem:points}(b), pick $d$ points $x_{d+1}, \dots, x_{2d}$ in $\R \setminus B$ on which
		evaluation is a bijection from $F'$ onto $\R^d$, and then $d$ points $x_1, \dots, x_d$ in
		$\R \setminus (B \cup \{x_{d+1}, \dots, x_{2d}\})$ with the same property. We write
		$P = \{1, \dots, d\}$ and $Q = \{d+1, \dots, 2d\}$ for the corresponding sets of indices. Since
		evaluation on $Q$ is injective on $F'$, $T^j$ is non-constant on $Q$ (otherwise $T^j - c$ would
		vanish in $F'$ for some constant $c$), and we pick $\kappa, \Lambda \in Q$ with
		$T^j_\kappa < T^j_\Lambda$. Similarly, by Lemma~\ref{lem:points}(a), $T^j$ is non-constant on
		$\R \setminus (B \cup \{x_1, \dots, x_{2d}\})$, so we can pick there two points $y', y''$ with
		$T^j(y') \neq T^j(y'')$. The last point $x_\ell$, $\ell = 2d + 1$, will be one of them.
		
		If $f''(1) \neq 0$, we take the trivial anchor $\theta^\ast = 0$ of Remark~\ref{rem:trivialanchor},
		with weights $\lambda_i = \epsilon$ for $i \in P \cup \{\ell\}$ and
		$\lambda_i = (1 - (d+1)\epsilon)/d$ for $i \in Q$. If $f''(1) = 0$, we apply Lemma~\ref{lem:anchor}
		to the points $x_1, \dots, x_{2d}, y', y''$, with coordinate $j$ and pivots $(\kappa, \Lambda)$; we
		then fix $\eta \in (0, \delta]$ and use the weights of Lemma~\ref{lem:anchor} with
		$A = P \cup \{\ell\}$ and $B = Q \setminus \{\kappa, \Lambda\}$. In both cases $\theta^\ast$, hence
		every $p_i$ and $w_i$, depends neither on $\epsilon$ nor on the choice of $x_\ell$, and as
		$\epsilon \to 0$ the weights $\lambda_i$, $i \in Q$, converge to positive limits $\lambda^0_i$.
		Consequently $\nu^\ast_i \to \nu^0_i \eqdef \lambda^0_i p_i > 0$ for $i \in Q$,
		$\nu^\ast_i \to 0$ for $i \in P \cup \{\ell\}$, and the barycentre (before translation) converges:
		\[
		\tau_j \;\longrightarrow\; \tau^0_j \eqdef
		\frac{\sum_{i \in Q} \nu^0_i\, w_i\, T^j_i}{\sum_{i \in Q} \nu^0_i\, w_i}\,,
		\]
		a number which does not depend on $x_\ell$. We choose $x_\ell \in \{y', y''\}$ with
		$T^j(x_\ell) \neq \tau^0_j$ and set $\X = \{x_1, \dots, x_{2d}, x_\ell\}$, so that $n = 2d + 1$. For
		$\epsilon$ small enough, these data form an anchored configuration on $\X$, and
		$T^j(x_\ell) \neq \tau_j$; after the standing translation, $T^j_\ell \neq 0$.
		
		Since $\X \subset \R \setminus B$ and evaluation on $Q$ is injective on $F'$, $\dim S_\X = d$ and
		the restriction maps $S_\X \to \R^P$ and $S_\X \to \R^Q$ are bijective. Prescribe
		\[
		r_i = s \quad (i \in P)\,, \qquad r_\ell = u\,,
		\]
		and let $g^1, \dots, g^d$ be a basis of $S_\X$. The condition $r \in \mathcal{A}(\nu_\ast)$ is the
		linear system
		\[
		\sum_{i \in Q} \nu^\ast_i\, g^m_i\, (r_i - 1) \;=\; - \sum_{i \in P \cup \{\ell\}} \nu^\ast_i\, g^m_i\, (r_i - 1)
		\qquad (m = 1, \dots, d)
		\]
		in the unknowns $(r_i - 1)_{i \in Q}$. As $\epsilon \to 0$, its matrix converges to
		$\bigl(\nu^0_i\, g^m_i\bigr)_{m,\, i \in Q}$, which is invertible since $(g^m_i)_{m,\, i \in Q}$ is,
		while its right-hand side is $O(\epsilon)$. Hence the solution exists for $\epsilon$ small and
		satisfies $r_i \to 1$ for $i \in Q$; in particular $r_i > 0$ and $r \in \mathcal{R}(\nu_\ast)$. We
		fix such an $\epsilon$; by Step~\ref{step:affine}, $\zeta^j(r) \in S_\X$. We can now conclude.
		
		\bigskip
		\noindent\emph{Case 1: $f''(1) = 0$.}
		\smallskip
		
		We show that this case cannot occur. Here $h'(1) = -f''(1) = 0$; take $s = 1$ in the
		construction above. At the indices $i \in P$, where $r_i = 1$, the vector
		$\zeta^j(r) = \bigl(w_i\, T^j_i\, h'(r_i)\bigr)_i$ vanishes; as $S_\X \to \R^P$ is injective,
		$\zeta^j(r) = 0$. At the index $\ell$ this reads
		\[
		w_\ell\, T^j_\ell\, h'(u) = 0\,,
		\]
		and $w_\ell > 0$, $T^j_\ell \neq 0$ give $h'(u) = 0$. As $u > 0$ was arbitrary, $h' \equiv 0$ on
		$\R^*_+$, so $f''(x) = -h'(x)/x$ vanishes identically and $f$ is affine (contradicting strict
		convexity). A $\mathcal{T}$-entropy-adapted $C^2$ divergence therefore satisfies $f''(1) \neq 0$,
		and Case~2 proves the theorem.
		
		\bigskip
		\noindent\emph{Case 2: $f''(1) \neq 0$.}
		\smallskip
		
		The anchor is the trivial one of Remark~\ref{rem:trivialanchor}: $\nu_\ast = \lambda$ and the
		weights are constant, $w_i \equiv w = 1/f''(1) > 0$. Since $T^j|_\X \in S_\X$, the vector
		\[
		\zeta^j(r) - w\, h'(s)\, T^j|_\X \;=\; \bigl(w\, T^j_i\, (h'(r_i) - h'(s))\bigr)_{1 \le i \le n}
		\]
		belongs to $S_\X$ and vanishes on $P$, where $r_i = s$; it is therefore zero. At the index $\ell$
		this gives
		\[
		w\, T^j_\ell\, \bigl(h'(u) - h'(s)\bigr) = 0\,,
		\]
		and $T^j_\ell \neq 0$ gives $h'(u) = h'(s)$. As $s, u > 0$ were arbitrary, $h'$ is constant on
		$\R^*_+$ and Lemma~\ref{lem:const} concludes: up to equivalence, $\Df = \KL$.
	\end{proof}
	
	\begin{remark}
		For $k = 1$ and a function $T$ taking each value at most twice, one has $B = \emptyset$ and $d = 2$,
		and the common configuration consists of five points: $P = \{1, 2\}$ carries the prescribed value
		$s$, $Q = \{3, 4\}$ consists of the two pivots used to balance the constraints, and $x_5$ is the
		test point. No reduction to families of the form $(1, T)$ is needed: such a reduction would require
		that a $\mathcal{T}$-entropy-adapted $f$ be $(1, T^j)$-entropy-adapted, which does not follow
		directly from Definition~\ref{def:adapted}.
	\end{remark}

	\section{Numerical Results}
	\label{s:Num}
	\subsection{Setting}
	
	We illustrate numerically the distinction between moment matching and best approximation within the generalized exponential family. The state space is
	\[
	\mathcal X=[-1,1],
	\]
	and the reference measure $\lambda$ is the uniform probability measure on $\mathcal X$, with density
	\[
	\xi(x)=\frac12.
	\]
	We consider the Chebyshev polynomials of degree $0$ to $k=10$,
	\[
	\mathcal{T}=(T^0,\ldots,T^k),
	\qquad T^0\equiv 1,
	\]
	as the free family defining the moment constraints.
	
	The target distribution is chosen to have a smooth, strictly positive and bimodal density:
	\[
	p(x)
	\propto
	0.05
	+0.75\exp\left(
	-\frac12\left(\frac{x+0.38}{0.12}\right)^2
	\right)
	+\exp\left(
	-\frac12\left(\frac{x-0.28}{0.18}\right)^2
	\right).
	\]
	The constant positive background ensures that the target density remains strictly positive on the whole interval.
	
	For a given generator $f$, we compare two elements of the associated family $\Ef$.
	The first one is the \emph{moment matcher}, denoted by $\nu^{\mathrm{MM}}_f$, obtained as the minimum-$D_f$ reconstruction under the moment constraints
	\[
	\nu^{\mathrm{MM}}_f(T^j)=\nu_0(T^j),
	\qquad j=0,\ldots,k.
	\]
	Equivalently, it is the generalized maximum-entropy reconstruction associated with $D_f$.
	
	The second one is the genuine $\Df$-projection of the target distribution onto $\Ef$,
	\[
	\nu^{\mathrm{Proj}}_f
	\in
	\argmin_{\nu_\theta\in\Ef}
	\Df(\nu_0,\nu_\theta).
	\]
	To quantify the discrepancy between these two constructions, we introduce the \emph{projection gap}
	\[
	\Delta_f
	:=
	\Df(\nu_0,\nu^{\mathrm{MM}}_f)
	-
	\Df(\nu_0,\nu^{\mathrm{Proj}}_f)
	\geq 0,
	\]
	as well as the $L^2$ distance
	\[
	\left\|
	p^{\mathrm{MM}}_f-p^{\mathrm{Proj}}_f
	\right\|_{L^2([-1,1])}.
	\]
	
	We consider two divergences. First, the Kullback--Leibler divergence, associated with
	\[
	f_{\mathrm{KL}}(u)=u\ln u-u+1,
	\]
	for which Theorem~\ref{thm:main} predicts that moment matching and projection coincide. We then consider the Pearson $\chi^2$ divergence, with generator
	\[
	f_{\chi^2}(u)=\frac12(u-1)^2,
	\]
	as a representative alternative $f$-divergence.
	
	All integrals and moments are evaluated numerically using a Gauss--Legendre quadrature of order $600$. The moment matcher is computed through the finite-dimensional dual problem, while the best projection is obtained by directly minimizing $\Df(\nu_0,\nu_\theta)$ over the normalized family $\Ef$.
	
	\subsection{Reconstruction}
	
	Figure~\ref{fig:numerical_projection_kl} shows the result for the Kullback--Leibler divergence. The moment-matching reconstruction and the direct KL projection are visually indistinguishable. Numerically, the projection gap is
	\[
	\Delta_{\mathrm{KL}}
	=
	6.373\times 10^{-12},
	\]
	while
	\[
	\left\|
	p^{\mathrm{MM}}_{\mathrm{KL}}
	-
	p^{\mathrm{Proj}}_{\mathrm{KL}}
	\right\|_{L^2}
	=
	1.484\times 10^{-6}.
	\]
	These residual values are at the level of the numerical optimization and quadrature errors. The experiment therefore recovers, up to numerical precision, the identity
	\[
	\nu^{\mathrm{MM}}_{\mathrm{KL}}
	=
	\nu^{\mathrm{Proj}}_{\mathrm{KL}},
	\]
	in agreement with the projection property established theoretically.
	
	\begin{figure}[t]
		\centering
		\includegraphics[width=0.82\textwidth]{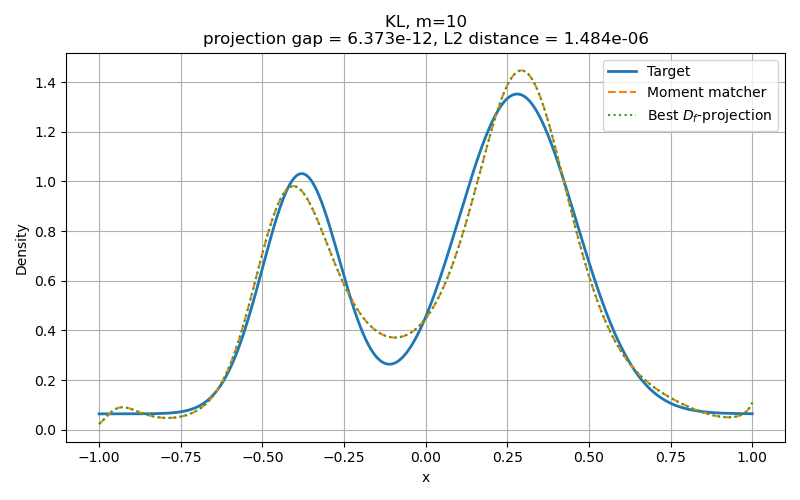}
		\caption{
			Moment matching and best projection for the Kullback--Leibler
			divergence with $k=10$ Chebyshev moments.
			The moment matcher and the best KL projection coincide up to numerical
			precision.
		}
		\label{fig:numerical_projection_kl}
	\end{figure}
	
	The situation is different for the Pearson $\chi^2$ divergence, as shown in Figure~\ref{fig:numerical_projection_chi2}. Although the moment matcher reproduces exactly the prescribed finite collection of moments, it does not minimize the Pearson divergence between the target and the associated generalized exponential family. In this case, we obtain
	\[
	\Delta_{\chi^2}
	=
	3.630\times 10^{-2},
	\]
	and
	\[
	\left\|
	p^{\mathrm{MM}}_{\chi^2}
	-
	p^{\mathrm{Proj}}_{\chi^2}
	\right\|_{L^2}
	=
	1.658\times 10^{-1}.
	\]
	The difference is clearly visible over the whole interval, and is particularly pronounced close to the boundary and around the two modes of the target density.
	
	\begin{figure}[t]
		\centering
		\includegraphics[width=0.82\textwidth]{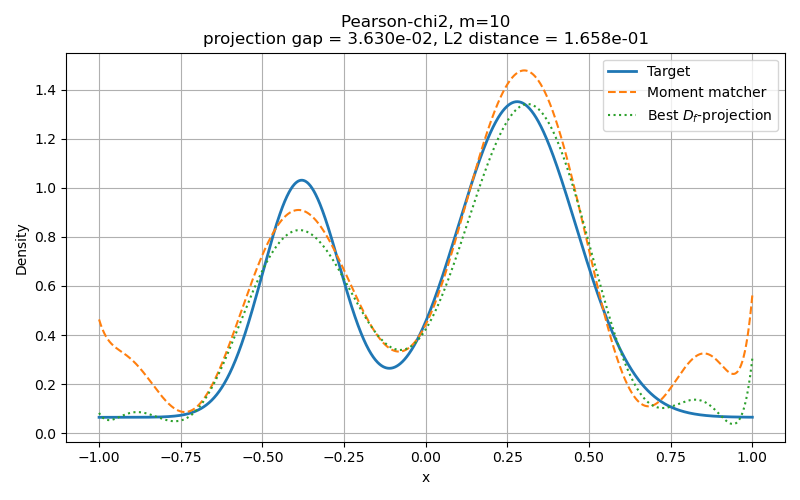}
		\caption{
			Moment matching and best projection for the Pearson $\chi^2$
			divergence with $k=10$ Chebyshev moments.
			Contrary to the KL case, the moment matcher and the best
			$D_f$-projection are distinct.
		}
		\label{fig:numerical_projection_chi2}
	\end{figure}
	
	These experiments provide a numerical illustration of the characterization established in Theorem~\ref{thm:main}:
	\begin{itemize}
		\item For KL, the information contained in the prescribed moments is sufficient to recover the best approximation of the target within the associated exponential family. 
		\item For the Pearson $\chi^2$ divergence, the generalized maximum-entropy reconstruction still matches exactly the same prescribed moments, but this information is no longer sufficient to identify the best $\Df$-approximation within $\Ef$. The non-zero projection gap therefore quantifies precisely the failure of the projection property outside the Kullback--Leibler setting.
	\end{itemize}
	
	\subsection{The entropy-adapted inequality: KL versus other $f$-divergences}
	\label{s:num-adapted}
	
	We now illustrate the entropy-adapted inequality of Definition~\ref{def:adapted} itself, by measuring
	how negative
	\[
	\Df(\mu,\nu_\theta)-\Df(\mu,\nu_{\bar\theta})
	\]
	can become when $\mu$ ranges over the moment set of $\nu_0$ and $\nu_\theta$ over $\Ef$. In this experiment
	we use fewer moments than above, namely the Chebyshev polynomials of degree $0$ to $k=6$. We draw $40$
	measures $\mu\in\MT{\nu_0}$ (including $\nu_0$ itself) of the
	form $\dd\mu/\dd\lambda=p_0+s\,r$, where $r$ is a random combination of
	the Chebyshev polynomials $T^7,\dots,T^{20}$ orthogonalised against $\mathrm{span}(\mathcal{T})$ in $L^2(\lambda)$,
	and $s>0$ is a scale small enough for this density to remain positive. All
	of them have exactly the moments of $\nu_0$.
	
	To vary $f$ continuously we use the
	$\alpha$-divergences
	\[
	f_\alpha(u)=\frac{u^\alpha-\alpha u+\alpha-1}{\alpha(\alpha-1)},\qquad
	(f_\alpha')^{-1}(t)=\bigl(1+(\alpha-1)t\bigr)^{1/(\alpha-1)},
	\]
	which give, by continuity at $\alpha\in\{0,1\}$, the reverse KL divergence ($\alpha=0$), the Hellinger divergence
	($\alpha=\tfrac12$, up to a factor $4$), KL ($\alpha=1$) and Pearson $\chi^2$
	($\alpha=2$). All generators satisfy $f(1)=f'(1)=0$. The moment matcher is computed
	by Newton's method on the convex dual
	$\theta\mapsto\int f^*(\theta^\top \mathcal{T})\,\dd\lambda-\theta^\top\nu(\mathcal{T})$. The projection
	$\argmin_{\Ef}\Df(\mu,\cdot)$ is computed over $(\theta_1,\dots,\theta_k)$,
	with $\theta_0$ fixed by normalisation and an exact gradient
	$\partial_{\theta_j}\Df(\mu,\nu_\theta)=\int h(\dd\mu/\dd\nu_\theta)\,\partial_{\theta_j}p_\theta\,\dd\lambda$,
	where $h(x)=f(x)-xf'(x)$ is the function of Lemma~\ref{lem:const}.
	For each of the $40$ measures $\mu$ and each $\alpha\in[0,2]$, we compute
	the relative projection gap
	\[
	\delta_\alpha(\mu)=\frac{D_{f_\alpha}(\mu,\nu_{\bar\theta})-\min_{\mathcal{E}_{f_\alpha}}D_{f_\alpha}(\mu,\cdot)}
	{D_{f_\alpha}(\mu,\nu_{\bar\theta})},
	\]
	which measures the worst violation of the inequality in Definition~\ref{def:adapted} for a given $\mu$
	(Figure~\ref{fig:num-C}). The gap is of order $10^{-2}$--$10^{-1}$ away from $\alpha=1$; it decreases
	continuously as $\alpha\to1$, as expected from the continuity of $f_\alpha$ in $\alpha$, and vanishes only
	at $\alpha=1$, where it is at machine precision (median $1.3\times10^{-16}$). Within this one-parameter
	family, KL is therefore the only divergence for which the projection property holds, as
	Theorem~\ref{thm:main} predicts.
	
	\begin{figure}[t]
		\centering
		\includegraphics[width=0.85\textwidth]{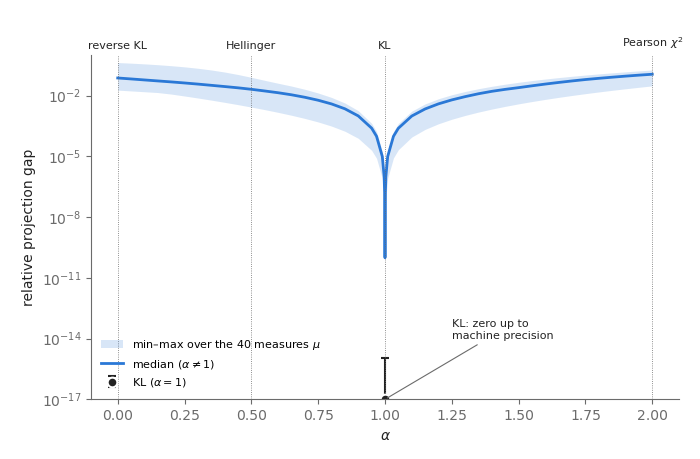}
		\caption{Relative projection gap $\delta_\alpha(\mu)$ for the
			$\alpha$-divergences, over $40$ measures $\mu$ sharing the same $k=6$ Chebyshev
			moments. It is zero only for KL.}
		\label{fig:num-C}
	\end{figure}
	
	\section{Conclusion}
	
	We studied whether the classical coincidence between maximum-entropy reconstruction, projection onto the associated exponential family, and ordinary moment matching, which makes Kullback--Leibler so effective for moment-based reconstruction, extends to other $f$-divergences. Introducing the notion of $\mathcal T$-entropy-adapted generators, we showed that KL is, up to equivalence, the only $C^2$ $f$-divergence for which this compatibility holds. The maximum-entropy construction itself survives for any $f$, but the Pythagorean mechanism that turns reconstruction into an approximation problem for the log-density is a specific feature of KL geometry. This explains notably why the guarantees of \citet{bringas2026momentmatching} for moment-based DistRL do not carry over to generalized entropies.
	
	Our result concerns the sharp projection inequality, with constant one. Whether a relaxed version
	with a multiplicative constant, $D_f(\mu,\nu_\theta) \geq c\, D_f(\mu,\nu_{\bar\theta})$ for some
	$c \in (0,1)$, can hold for $f$-divergences other than KL remains open. Indeed, such a quasi-projection
	property would still suffice for many applications, e.g. to derive convergence rates for the
	moment-matching reconstruction from those of the best approximation in $\Ef$.
	
	\section*{Acknowledgments}
	The authors acknowledge the support of Chaire SeqALO (ANR-20-CHIA-0020-01), and PEPR-IA project FOUNDRY  (ANR-23-PEIA-0003).
	
	\bibliographystyle{plainnat}
	\bibliography{bibliography}
	
	\appendix
	\section{Relation with the Bregman characterization of KL}
	\label{s:AppA}
	It is well known that, on probability distributions, the Kullback--Leibler divergence is, up to the usual
	equivalence of generators, the only divergence that belongs simultaneously
	to the classes of $f$-divergences and Bregman divergences
	\citep{amari2009alpha}; on positive measures, the $\alpha$-divergences share this property. Since Bregman divergences satisfy natural Pythagorean
	projection identities, it is therefore legitimate to ask whether the
	characterization established in Theorem~\ref{thm:main} is merely a direct
	consequence of this classical result. We explain here why this is not the
	case. An illustration is provided in Figure~\ref{fig:bregman-fdiv-summary}.
	
	Let
	\[
	B_F(p,q)
	=
	F(p)-F(q)-\langle \nabla F(q),p-q\rangle
	\]
	be a Bregman divergence. Its three-point identity reads
	\[
	B_F(p,r)
	=
	B_F(p,q)+B_F(q,r)
	+
	\langle p-q,\nabla F(q)-\nabla F(r)\rangle.
	\]
	Consider a family that is affine in the dual coordinates induced by $F$,
	namely
	\[
	\nabla F(p_\theta)
	=
	\theta^\top \mathcal{T}
	=
	\theta_0 + \sum_{j=1}^k \theta_j T^j,
	\]
	with $\theta_0$ fixed by normalization. If $p_{\bar\theta}$ matches the moments of
	$p$, then
	\[
	\langle p-p_{\bar\theta},
	\nabla F(p_{\bar\theta})-\nabla F(p_\theta)\rangle
	=0,
	\]
	because the second factor belongs to $\spn(\mathcal{T})$, whereas
	$p-p_{\bar\theta}$ is orthogonal to this space by moment matching (the constant $T^0$ included,
	both being probability densities).
	Consequently,
	\[
	B_F(p,p_\theta)
	=
	B_F(p,p_{\bar\theta})
	+
	B_F(p_{\bar\theta},p_\theta),
	\]
	and hence
	\[
	B_F(p,p_{\bar\theta})
	\leq
	B_F(p,p_\theta).
	\]
	Thus, the projection property studied in this paper is indeed naturally
	satisfied by Bregman divergences when the approximation family is chosen
	according to their dual geometry. This is the inclusion of the set of Bregman divergences in the set of entropy-adapted divergences in Figure~\ref{fig:bregman-fdiv-summary}.
	
	For the $f$-divergences considered here, the associated maximum-entropy
	family has precisely this form for the Bregman divergence generated by the
	integral functional
	\[
	F_f(p)=\int_{\X} f(p)\,\dd\lambda.
	\]
	Indeed,
	\[
	\nabla F_f(p)=f'(p),
	\]
	and the family introduced in Section~\ref{s:Preliminaries} satisfies
	\[
	f'(p_\theta)=\theta^\top \mathcal{T}.
	\]
	The corresponding Bregman divergence is
	\[
	B_{F_f}(p,q)
	=
	\int_{\X}
	\left[
	f(p)-f(q)-f'(q)(p-q)
	\right]\dd\lambda.
	\]
	For this divergence, moment matching therefore yields the Pythagorean
	identity above automatically.
	
	However, the divergence studied in this paper is not $B_{F_f}$ but the $f$-divergence
	\[
	D_f(p,q)
	=
	\int_{\X}
	q\,f\!\left(\frac{p}{q}\right)\dd\lambda.
	\]
	These two divergences are different in general. They coincide, up to a
	multiplicative constant, in the Kullback--Leibler case, but not for a
	generic generator $f$. Our characterization therefore asks when the
	$f$-divergence $D_f$ inherits the projection geometry that is naturally
	associated with the Bregman divergence $B_{F_f}$.
	
	More importantly, the projection inequality alone does not imply that a
	divergence is Bregman. Indeed, let
	\[
	D(p,q)=\KL(p,q)^2.
	\]
	Since the map $x\mapsto x^2$ is strictly increasing on $\mathbb R_+$,
	$D$ has exactly the same minimizers as $\KL$. Hence, whenever
	$p_{\bar\theta}$ is the moment-matching point in an exponential family,
	\[
	D(p,p_{\bar\theta})
	\leq
	D(p,p_\theta)
	\qquad
	\text{for every }p_\theta
	\]
	follows immediately from the corresponding KL inequality. Nevertheless,
	$D=\KL^2$ is not a Bregman divergence. For a Bregman divergence,
	the Hessian with respect to its first argument necessarily satisfies
	\[
	\nabla_{pp}^2 B_F(p,q)=\nabla^2F(p),
	\]
	and is therefore independent of $q$. In contrast, writing
	$K(p,q)=\KL(p,q)$ gives
	\[
	\nabla_{pp}^2 K(p,q)^2
	=
	2\nabla_pK(p,q)\nabla_pK(p,q)^\top
	+
	2K(p,q)\nabla_{pp}^2K(p,q),
	\]
	which depends explicitly on $q$.
	
	Therefore, the moment-matching projection property does not imply that the divergence is Bregman.
	As a consequence, Theorem~\ref{thm:main} cannot be
	obtained simply by combining the projection property with the classical
	characterization that the only Bregman $f$-divergence is KL.
	The content of Theorem~\ref{thm:main} is that the notion of entropy-adapted divergence is already strong
	enough to force
	\[
	x f''(x)=\mathrm{constant},
	\]
	and hence to characterize the Kullback--Leibler divergence.
	
\end{document}